\documentclass[10pt]{article}
\usepackage{delarray} 
\usepackage{amscd}
\usepackage{amssymb}
\usepackage{amsmath}
\usepackage{amsthm,amsxtra}
\usepackage{latexsym,amsmath,mathrsfs}
\usepackage[bookmarks,colorlinks]{hyperref}
\usepackage[all]{xy}
\usepackage{amsfonts}
\usepackage{color}
\usepackage{fancybox}
\usepackage{colordvi}
\usepackage{multicol}
\usepackage{textcomp}
\usepackage{stmaryrd}
\usepackage[dvips]{graphicx}

\usepackage{enumerate,xspace}
\usepackage{geometry}
\usepackage{tocbibind}

    \newcommand{\BQ}{\mathbb{Q}}

    \newcommand{\Cl}{\mathrm{Cl}}

    \newcommand{\fo}{\mathfrak{o}}
    \newcommand{\fp}{\mathfrak{p}}

\theoremstyle{plain}

    \newtheorem{thm}{Theorem}[section] \newtheorem{cor}[thm]{Corollary}
    \newtheorem{lem}[thm]{Lemma} 
    \newtheorem{prop}[thm]{Proposition}
    \newtheorem {conj}[thm]{Conjecture}

    \theoremstyle{definition}

    \newtheorem{defn}[thm]{Definition}

    \theoremstyle{remark}
    \newtheorem {rem}[thm]{Remark}
    \newtheorem {example}[thm]{Example}
    
    \newtheorem {condition}[thm]{Condition}

    \numberwithin{equation}{section}

\begin{document}

\title{Hilbert modular forms and class numbers}
\author{Qinyun Tan\\
\small School of Mathematical Sciences, East China Normal
University, Shanghai, China \\
\small52265500010@stu.ecnu.edu.cn \\
 Bingyong Xie \footnote{This paper is
supported by the National Natural Science Foundation of China (grant
12231001), and by Science and Technology Commission of Shanghai
Municipality (no. 22DZ2229014).The authors are supported by
Fundamental Research
Funds for the Central Universities.} \\
\small School of Mathematical Sciences, East China Normal
University, Shanghai, China \\ \small byxie@math.ecnu.edu.cn}
\date{}
\maketitle

\begin{abstract} In 1975, Goldfeld gave an effective solution to
Gauss's conjecture on the class numbers of imaginary quadratic fields.
In this paper, we generalize Goldfeld's theorem to the setting of
totally real number fields.
\end{abstract}

\section*{Introduction}

In his famous book {\it Disquisitiones arithmeticae}, Gauss proposed
three conjectures on ideal class numbers of quadratic fields. One of
them can be stated as follows.

\begin{conj}\label{conj:Gauss} For an imaginary quadratic field $K$, let $D_K$ and $h(K)$
	denote its discriminant and class number respectively. Then
	$h(K)\rightarrow \infty$ when $D_K\rightarrow \infty$.
\end{conj}

This conjecture was proved by Hecke, Deuring, Mordell and Heilbronn.

\begin{thm}\label{thm:not-effective} $($Hecke-Deuring-Heilbronn \cite{Landau, MD33, H}$)$ Let $K$ run
	through all imaginary quadratic fields. If $D_K\rightarrow \infty$,
	then $h(K)\rightarrow \infty$.
\end{thm}

This theorem is not effective.

In 1975, Goldfeld \cite{G75} effectively solved Conjecture
\ref{conj:Gauss} assuming the existence of an elliptic curve over
$\BQ$ with analytic rank $3$. Such an elliptic curve $E$ was later found
by Gross and Zagier \cite{GZ86}.

\begin{thm}\label{thm:effective} $($Goldfeld-Gross-Zagier$)$ Let $K$ be an imaginary
	quadratic field, and let $p_{0}$ be the maximal prime divisor of $d(K)$. Then
	$$ h(D_K) > \frac{1}{7000} (\log |D_K|) \prod_{p\mid d(K),\ p\ne p_{0}} \left(1-\frac{[2\sqrt{p}]}{p+1}\right),
	$$ where $p$ denotes a prime number, and $[x]$ denotes the integer
	part of $x$.
\end{thm}

The reader may consult \cite{G85} for the history on Conjecture
\ref{conj:Gauss} and its solutions.

The goal of the present paper is to generalize  Theorem
\ref{thm:effective} to the setting of totally real number fields.

A related result was obtained by Sunley \cite{Sun73}, who gave a
generalization of a result of Tatuzawa \cite{Tat51}.

\begin{thm} $($Sunley$)$ Let $F$ be a fixed totally real algebraic number field and let
	$K$ be a totally imaginary quadratic extension of $F$ having
	conductor $\mathfrak{f}_K$ and class number $h$. Then there exists
	an effectively computable constant $c=c(F, h)$ such that
	$N\mathfrak{f}_K < c$, with the possible exception of one field $K$.
\end{thm}

Our main result is the following

\begin{thm} \label{thm:main0}  Let $F$ be a fixed totally real number
	field that is Galois over $\BQ$ with solvable Galois group. When
	$[F:\BQ]$ is even, we further demand that the number of places of
	$F$ above $37$ is even.
	
	Then for any $\epsilon>0$, there exists an effective constant $C(\epsilon)>0$ such that, for any totally imaginary quadratic extension $K$ of
	$F$ we have
	$$h_{K}\ge C(\epsilon) \cdot ({\rm log}|\mathfrak{d}_{K/F}|_{\mathbb{R}})^{1-\epsilon}. $$
\end{thm}

Suppose $F=\mathbb{Q}$. Assuming the existence of elliptic curve
with algebraic rank three, Goldfeld constructed two Dirichlet series
with certain properties, and then used them to give both upper and
lower bounds of a key integral $J(U)$. This integral $J(U)$ relates
$h_{K}$ and ${\rm log}|\mathfrak{d}_{K/F}|$, which leads to the
lower bound for $h_{K}$ in Theorem \ref{thm:effective}. In \cite{JO
	84} Oesterl\'{e} simplified Goldfeld's original proof.

Our proof of Theorem \ref{thm:main0} is similar to Oesterl\'{e}'s.

When $F$ is a solvable extension of $\mathbb{Q}$, we use the tool of
base change to lift the elliptic curve constructed by Gross-Zagier
(in fact the weight $2$ modular from attached to this elliptic
curve) to a Hilbert modular form $\mathbf{f}$ of parallel weight
$2$. Then we use this Hilbert modular from to construct two
Dirichlet series $\varPsi(s)$ and $\varPhi(s)$, and use them to
bound $|J(U)|$.

The lower bound for $|J(U)|$ follows from complex analysis. To give
its upper bound, the key point is to control $\varPsi(s)$ and
$\varPhi(s)$ by using Dirichlet series $\zeta_{F}(s)$ and
$\frac{\zeta_{K}(s)}{\zeta_{F}(2s)}$. To estimate $\zeta_{F}(s)$ and
$\frac{\zeta_{K}(s)}{\zeta_{F}(2s)}$, we prove some estimations on
the number of (certain) elements of $K$ (or $F$) with bounded norms.
Besides these estimations, we also need the following inequality
\begin{eqnarray}\label{eq:hk and t}
	h_{K}\ge \frac{2^{t+[F:\mathbb{Q}]-1}}{[\mathfrak{o}^{\times}_{F}:(\mathfrak{o}^{\times}_{F})^{2}]}.
\end{eqnarray}
Here $t$ is the number of different prime divisors of
$\mathfrak{d}_{K/F}$.

To build (\ref{eq:hk and t}) we need a theory of quadratic forms and
a reduction theory about these quadratic forms. When $F=\mathbb{Q}$,
such a reduction theory is due to Gauss. When $F$ is general, one
meets the difficulty that $h_F$ may be larger than $1$. We build a
theory of quadratic form on locally free $\mathfrak{o}_F$-modules
$M$ of rank $2$, instead of free modules. One difficulty is to
define the discriminant of a quadratic form. Luckily, for each finite
place $v$ of $F$, $M_v=M\otimes_{\mathfrak{o}_F}\mathfrak{o}_{F_v}$
is  free, so that the discriminant of $M_v$ can be defined. Then we
patch them to obtain an ideal of $\mathfrak{o}_F$. We regard this
ideal, instead of a number, as the discriminant of $M$. Another
difficulty is the definition of fundamental quadratic forms. We
again use the localization method to deal with this problem. Basing
on the reduction theory we show the number of (weak) equivalence classes is
at least $
\frac{2^{t+[F:\mathbb{Q}]-1}}{[\mathfrak{o}^{\times}_{F}:(\mathfrak{o}^{\times}_{F})^{2}]}$,
which deduces (\ref{eq:hk and t}).

Our paper is organized as follows. Section \ref{sec:quad-form} is
devoted to quadratic forms on $\mathfrak{o}_F$-modules rank $2$. In
Section \ref{sec:Gross-Zagier} we lift the elliptic curve with
analytic rank $3$ constructed by Gross and Zagier (more precisely,
the modular form attached to this elliptic curve) to solvable
totally real fields $F$, and use it to construct two Dirichlet
series ${\it \Phi}(s)$ and ${\it \Psi}(s)$ that satisfy certain
properties. Section \ref{sec:esti} is devoted to some estimates
based on ${\it \Phi}(s)$ and ${\it \Psi}(s)$. In Section
\ref{sec:proof} these estimates are used to prove Theorem
\ref{thm:main0}.

\section*{Notations}

Let $F$ be a totally real number field,  $\fo=\fo_F$ its ring of
integers.
Write
$\mathfrak{o}^*=\mathfrak{o} \backslash \{0\}$.
An element $a$ of $F$ is called {\it totally positive} if it is
positive for each embedding of $F$ into $\mathbb{R}$.

For a fractional ideal $\mathfrak{a}$ of $F$, we simply use
$|\mathfrak{a}|_\mathbb{R}$ to denote the real norm of
$N_{F/\BQ}(\mathfrak{a})$.

Let $\mathrm{Cl}_F$ and $\mathrm{Cl}_K$ be the groups of ideal
classes of $F$ and $K$ respectively. Put $h_K=\sharp(\mathrm{Cl}_F)$
and $h_F=\sharp(\mathrm{Cl}_F)$.

\section{Quadratic forms on $\mathfrak{o}$-modules}
\label{sec:quad-form}

\subsection{Definite quadratic forms}

Let $E$ be a domain. By a {\it (binary) quadratic form} $(M,Q)$ over $E$, we
mean that $M$ is a projective $E$-module of rank $2$, and $Q$ is a
symmetric bilinear maps
$$ Q: M\times M \rightarrow Frac(E). $$ Here $Frac(E)$ denotes the fractional field of $E$. We also say that $Q$ is a
quadratic form on $M$. We associate to $Q$ a quadratic function
$q_Q$ on $M$ by $ q_Q(m) =Q(m,m), $ ($m\in M$).

We say that two quadratic forms $(M_1, Q_1)$ and $(M_2, Q_2)$ over
$E$ are {\it equivalent} if  there exists a map $\phi:
M_1\rightarrow M_2 $ such that $Q_1=Q_2\circ (\phi\times \phi)$.

When $M$ is free, if $\{\alpha, \beta\}$ is a basis of $M$, there
exists $a$, $b$, $c\in Frac(E)$ such that
\begin{eqnarray*}  q_{\alpha, \beta}(x,y):=q_Q(x \alpha+ y\beta )& =
& ax^2+bxy+cy^2 \\ & = &
\left(\begin{array}{c} x \\
y\end{array}\right)\left[\begin{array}{cc} a & \frac{b}{2} \\
\frac{b}{2}  & c \end{array}\right] \left(\begin{array}{cc} x &
y\end{array}\right) \hskip 15pt (x, y\in E).  \end{eqnarray*} Put
$d_Q=b^2-4ac$.

\begin{lem}\label{lem:basis-to-basis}
If $\{\alpha',\beta'\}$ is another basis of $M$ with $$
q_{\alpha',\beta'}(x, y )
=\left(\begin{array}{c} x \\
y\end{array}\right)\left[\begin{array}{cc} a' & \frac{b'}{2} \\
\frac{b'}{2}  & c' \end{array}\right] \left(\begin{array}{cc} x &
y\end{array}\right) \hskip 15pt (x, y\in E),$$  then there exists a
matrix $T\in \mathrm{GL}_2(E)$ such that
$$ \left [\begin{array}{cc} a' & \frac{b'}{2} \\
\frac{b'}{2}& c' \end{array}\right] = T^t \left [\begin{array}{cc} a & \frac{b}{2} \\
\frac{b}{2}& c \end{array}\right] T ,$$ and there exists a unit
$u\in E^\times$ such that $b'^2-4a'c'= u^2( b^2-4ac )$.
\end{lem}
\begin{proof} This is an exercise of linear algebra. For the second
relation we take $u=\det(T)$.
\end{proof}

Let $(M,Q)$ be a  quadratic form over $\fo$. We extend $Q$
 to a quadratic form $Q_F$ on $M_F=M\otimes_\fo F$. We can
attach to $Q$ a field $F(Q):=F( \sqrt{d_{Q_F}})$.

For each real place $\iota: F\rightarrow \mathbb{R}$, we can extend
$Q_F$ to $$  Q_\iota: M_F\otimes_{\iota(F)} \mathbb{R}
\times M_F\otimes_{\iota(F)} \mathbb{R} \rightarrow \mathbb{R}.
$$ We say that $Q$ is {\it $($positive$)$ definite} if  $Q_\iota$ are all
(positive) definite.

\begin{lem}\label{lem:positive} If $Q$ is definite, then  $d_{Q_F}$ is totally negative. In particular $d_{Q_F}
\neq 0 $.
\end{lem}
\begin{proof} By the theory of real quadratic forms, since $Q_\iota$
is definite, it has negative discriminant.
\end{proof}

By Lemma \ref{lem:positive}, if $Q$ is definite, then $F(Q)$ is a
(totally) imaginary quadratic extension of $F$.

\begin{lem}\label{lem:crit} \begin{enumerate}
\item $Q$ is positive definite if and only if
$q_Q(m)$ is totally positive for any nonzero $m\in M$.
\item $Q$ is definite if and only if $q_Q(m)$ $(m\in M, m\neq 0)$ has the same sign for each real embedding.
\end{enumerate}
\end{lem}
\begin{proof} The ``only if'' parts are clear.
For the ``if'' parts, we assume that $q_Q(m)\neq 0$ when $m\neq 0$.
We can take $m,n$ such that $Q(m,n)=0$. Then $m$ and $n$ are
linearly independent over $\fo$. For each real embedding $\iota$, we
have $q_{Q_{\iota}}(xm+yn)=x^2\iota(q_Q(m))+y^2\iota(q_Q(n))$. Since
$\iota(q_Q(m))$ and $\iota(q_Q(n))$ have same sign, when $xm+yn\neq
0$, $q_{Q_{\iota}}(xm+yn)$ has the same sign as $\iota(q_Q(m))$.
This proves the ``if'' parts.
\end{proof}

We define the discriminant of $Q$. For any place $\fp$ of $F$, put
$M_\fp=M\otimes_{\fo}\fo_\fp$. Localizing at $\fp$ we obtain a
symmetric bilinear map $ Q_\fp: M_\fp\times M_\fp \rightarrow
F_{\mathfrak{p}}. $  Note that $M_\fp$ is free of rank $2$ over $\fo_{\fp}$.
So $d_{Q_\mathfrak{p}}$ makes sense.  Put $\mathrm{ind}_\fp Q=v_\fp(
d_{Q_\mathfrak{p}}) $.

\begin{lem} For almost all $\fp$, $\mathrm{ind}_\fp Q=0$.
\end{lem}
\begin{proof} Take two elements $m$, $n$ of
$M$ such that for almost all $\fp$, $\{m,n\}$ is a basis of $M_\fp$ over $\fo_{\mathfrak{p}}$. For
such a $\fp$,
$$\mathrm{ind}_\fp Q=v_\fp(4 (Q(m,m) Q(n,n) - Q(m,n)^2))= 2v_\fp(2)+
v_\fp( Q(m,m) Q(n,n) - Q(m,n)^2).$$ As $Q(m,m) Q(n,n) -
Q(m,n)^2\in\fo$ is nonzero, for almost all $\fp$,
$$v_\fp(Q(m,m) Q(n,n) - Q(m,n)^2)=0,$$ as expected.
\end{proof}

We define the {\it discriminant} of $Q $, denoted by $\mathfrak{d}_Q
$, to be the ideal $ \prod\limits_{\fp} \fp^{\mathrm{ind}_\fp Q}$.

\begin{lem} \label{lem:quad-1}
Let $M_1$ and $M_2$ be two projective $\fo$-modules locally free of
rank $1$, $Q $ a  definite quadratic form on $M_1\oplus M_2$. If
$\mathfrak{a}$ and $\mathfrak{b}$ are two ideals of $\fo_F$, then
$\mathfrak{d}_{Q|_{\mathfrak{a}M_1\oplus \mathfrak{b}M_2}}
=(\mathfrak{a}\mathfrak{b})^2 \mathfrak{d}_Q$.
\end{lem}
\begin{proof} This is obvious.
\end{proof}

If $N$ is a submodule of $M$ locally of rank $1$, we use $I(Q, N)$ to
denote the (fractional) ideal generated by the set $\{q_Q(n): n\in N\}$. If $N$
is free with a generator $n_0$, then $I(Q, N)$ is the principal
ideal $(q_Q(n_0))$.

\begin{lem} \label{lem:quad-2}
We have $I(Q, \mathfrak{a} N) = \mathfrak{a}^2 I(Q, N)$.
\end{lem}
\begin{proof} This follows from the relation $q_Q(an)=a^2q_Q(n)$, where $a\in \mathfrak{a}$ and $n\in N$.
\end{proof}

\begin{lem}\label{thm:only-one}
Let $Q$ be a definite quadratic form on $M$ with discriminant
$\mathfrak{d}$. Then there exists at most one saturated submodule
$N$ of $M$ that is locally free of rank $1$ such that $$ |I(Q,
N)|_{\mathbb{R}} <
\frac{|\mathfrak{d}_Q|_\mathbb{R}^{1/2}}{2^{[F:\BQ]}}.$$
\end{lem}
\begin{proof} Assume that there are two different saturated submodules $N_1$ and $N_2$ of $M$
of rank $1$ such that $$ |I(Q, N_1)|_{\mathbb{R}} \ < \
\frac{|\mathfrak{d}|_\mathbb{R}^{1/2}}{2^{[F:\BQ]}} , \hskip 10pt
|I(Q,  N_2)|_{\mathbb{R}} \ < \
\frac{|\mathfrak{d}|_\mathbb{R}^{1/2}}{2^{[F:\BQ]}}.$$ Put
$M'=N_1\oplus N_2$. Then $ \mathfrak{d}'= \mathfrak{d}_{ Q|_{M'}}  $
is a sub-ideal of $\mathfrak{d} $.

Take any two nonzero elements $n_1$ and $n_2$ in $N_1$ and $N_2$,
respectively. Assume that $\fo_{F }n_1=\mathfrak{a}N_1$ and $\fo_{F
} n_2=\mathfrak{b} N_2$.

Put $M''=\fo_{F }n_1 \oplus \fo_{F } n_2$ and
$\mathfrak{d}''=\mathfrak{d}_{Q|_{M''}}$. By Lemma \ref{lem:quad-1},
$\mathfrak{d}''=(\mathfrak{ab})^2 \mathfrak{d}'$. By Lemma
\ref{lem:quad-2}, $I(Q,   \fo_{F } n_1)= \mathfrak{a}^2 I(Q, N_1)$
and $I(Q,  \fo_{F } n_2) = \mathfrak{b}^2 I(Q, N_2)$. So we have
$$ |I(Q,  \fo_{F } n_1)|_{\mathbb{R}} = |\mathfrak{a}|_{\mathbb{R}}^2 |I(Q, N_1)|_{\mathbb{R}}
<
\frac{|\mathfrak{a}|^2_\mathbb{R}|\mathfrak{d}|_\mathbb{R}^{1/2}}{2^{[F:\BQ]}}
$$ and $$ |I(Q, \fo_{F} n_2)|_{\mathbb{R}} =
|\mathfrak{b}|_{\mathbb{R}}^2 |I(Q,  N_2)|_{\mathbb{R}} <
\frac{|\mathfrak{b}|^2_\mathbb{R}|\mathfrak{d}|_\mathbb{R}^{1/2}}{2^{[F:\BQ]}}
. $$

Assume that with respect to the basis $\{n_1, n_2\}$ of $N''$,
$Q|_{N''}$ is of the form $Q(xn_1+yn_2)=ax^2+bxy+cy^2$.  Then
$$ \mathfrak{d}'' =(b^2-4ac), \ \  I(Q, \fo_{F_\fp} n_1) =(a), \ \   I(Q, \fo_{F_\fp} n_2) =
(c) .$$ So
$$ |ac|_{\mathbb{R}} = |I(Q, \fo_{F_\fp} n_1)|_{\mathbb{R}} \cdot |I(Q, \fo_{F_\fp} n_2)|_{\mathbb{R}}
< \frac{|(\mathfrak{ab})^2\mathfrak{d}|_{\mathbb{R}}}{4^{[F:\BQ]}}
$$ and $$|4ac|_{\mathbb{R}} < |(\mathfrak{ab})^2\mathfrak{d}|_{\mathbb{R}} \leq |(\mathfrak{ab})^2\mathfrak{d}'|_{\mathbb{R}}
=|\mathfrak{d}''|_{\mathbb{R}}=|4ac-b^2|_{\mathbb{R}}. $$

Since $Q$ is definite, $4ac-b^2$ is totally positive. Since $b^2$ is
also totally positive, we have
$$|4ac-b^2|_{\mathbb{R}} \leq  |4ac|_{\mathbb{R}} ,$$ a
contradiction.
\end{proof}

\subsection{Fundamental quadratic forms}

\begin{defn}\label{def:fund} Let $d$ be an element of $\mathfrak{o}_\mathfrak{p}$.
\begin{enumerate}
\item In the case of $\mathfrak{p}\nmid 2$, if $d\notin\mathfrak{p}^2$, we call $d$ {\it fundamental}.
\item In the case of $\mathfrak{p}|2$, we call $d$ {\it fundamental} if $d$ is a quadratic residue $\mathrm{mod} \ 4$, and if either of the following two conditions holds:
\begin{quote}
 $\bullet$  $\pi_\mathfrak{p}^2\nmid d$, \\
 $\bullet$  $\pi_\mathfrak{p}^2| d $ and $\frac{d}{\pi_\mathfrak{p}^2}$ is not a quadratic residue modulo $4$.
\end{quote} Here $\pi_\mathfrak{p}$ is a uniformizer of
$\fo_\fp$.
\end{enumerate}
\end{defn}

\begin{lem}\label{lem:crit:fund} Let $d$ be an element of $ \mathfrak{o}_\mathfrak{p}$. Put $E=F_\mathfrak{p}(\sqrt{d})$.
\begin{enumerate}
\item\label{it:degree-1} If $E=F_\mathfrak{p}$, then $d$ is fundamental if and only if $d=u^2$, where $u\in \mathfrak{o}_\mathfrak{p}^\times$.
\item\label{it:degree-2} If $[E:F_\mathfrak{p}]=2$, writing
$\mathfrak{o}_E= \mathfrak{o}_\mathfrak{p}\oplus
\mathfrak{o}_\mathfrak{p} \omega$, and letting $d_\omega$ be the
discriminant for the basis $\{1, \omega\}$, then $d_\omega$ is
fundamental, and $d$ is fundamental if and only if $d=u^2d_\omega$
for some $u\in\mathfrak{o}_\mathfrak{p}^\times$.
\end{enumerate}
\end{lem}

\begin{proof} If $E=F_\mathfrak{p}$, then $d=u^2$ for some $u\in \mathfrak{o}_\mathfrak{p}$.
When $\mathfrak{p}\nmid 2$, by definition, $d\notin \mathfrak{p}^2$
and so $u\notin \mathfrak{p}$. When $\mathfrak{p}| 2$, and if $u\in
\mathfrak{p}$, then $d\in \mathfrak{p}^2$. By definition
$\frac{d}{\pi_\mathfrak{p}^2}=(\frac{u}{\pi_\mathfrak{p}})^2$ is not
a quadratic residue modulo $4$, a contraction. This proves
(\ref{it:degree-1}).

The argument for (\ref{it:degree-2}) is similar to the proofs of
\cite[Lemma 2.2, Lemma 2.3]{ZK 21}.
\end{proof}

\begin{cor}\label{cor: prime power in discriminant }
	Let $K/F$ be as usual. For any prime $\mathfrak{p}\mid\mathfrak{d}_{K/F}$, let $e_{\mathfrak{p}}$ be its ramification index over $\mathbb{Q}$. Then $v_{\mathfrak{p}}(\mathfrak{d}_{K/F})=1$ if $\mathfrak{p}\nmid2$, and $v_{\mathfrak{p}}(\mathfrak{d}_{K/F})\le 2e_{\mathfrak{p}}+1$ if $\mathfrak{p}\mid2$. 
\end{cor}

\begin{proof}
Let $\mathfrak{P}$ be the prime of $K$ lying over $\mathfrak{p}$ and we write $\mathfrak{o}_{K_\mathfrak{P}}=\mathfrak{o}_{F_\mathfrak{p}}\oplus\mathfrak{o}_{F_\mathfrak{p}}w$.
Then $K_{\mathfrak{P}}=F_{\mathfrak{p}}(\sqrt{d_{w}})$.  By Lemma \ref{lem:crit:fund} (\ref{it:degree-2}), $d_{w}$ is fundamental and generates
$\mathfrak{d}_{K_{\mathfrak{P}}/F_{\mathfrak{p}}}$. Thus when $\mathfrak{p}\nmid2$, we have $\mathfrak{p}\parallel d_{\omega}$ . If $\mathfrak{p}\mid2$,
let $\pi_{\mathfrak{p}}$ be a uniformizer of $\mathfrak{o}_{F_\mathfrak{p}}$. Then $d_{w}/\pi_{\mathfrak{p}}^{2}$ is not a quadratic residue modulo 4.
In particular, $4\nmid(d_{w}/\pi_{\mathfrak{p}}^{2})$ i.e $v_{\mathfrak{p}}(d_{w})< v_{\mathfrak{p}}(4)+2=2e_{\mathfrak{p}}+2$.	
\end{proof}

Let $(M,Q)$ be a quadratic form over $\fo$. Let $\mathfrak{n}_Q$ be
the ideal of $\mathfrak{o}$ generated by $\{q_Q(x):x\in M\}$.

\begin{defn}\label{defn:fund-quad} We call a quadratic form $(M,Q)$ over $\fo$ {\it fundamental}, if for each finite place $\mathfrak{p}$ of $F$,
$\frac{d_{Q_\mathfrak{p}}}{n^2_{Q_\mathfrak{p}}}$ is fundamental,
where $n_{Q_\mathfrak{p}}$ is a generator of
$\mathfrak{n}_{Q_\mathfrak{p}}$. Obviously, whether $(M,Q)$ is
fundamental or not does not depend on the choices of
$n_{Q_\mathfrak{p}}$.
\end{defn}

\begin{prop} \label{prop:fund} $(M,Q)$ is fundamental  if and only if $\mathfrak{d}_{F(Q)/F}=\mathfrak{d}_Q\mathfrak{n}_Q^{-2}$.
\end{prop}
\begin{proof} The ``only if'' part follows from Lemma \ref{lem:crit:fund}.

Next we prove the ``if'' part. For each finite place $\mathfrak{p}$
of $F$, let $n_{Q_\mathfrak{p}}$ be a generator of
$\mathfrak{n}_{Q_\mathfrak{p}}$. The relation
$\mathfrak{d}_{F(Q)/F}=\mathfrak{d}_Q\mathfrak{n}_Q^{-2}$ says that
\begin{equation}\label{eq:fund}
v_\mathfrak{p}\left(\frac{d_{Q_\mathfrak{p}}}{n_{Q_\mathfrak{p}}^2}\right)=v_\mathfrak{p}(\mathfrak{d}_{F(Q)/F}).
\end{equation}

When $\mathfrak{p}$ is split in $K=F(Q)$, $d_{Q_\mathfrak{p}}$ is a
square, and thus $\frac{ d_{Q_\mathfrak{p}} }{ n_{Q_\mathfrak{p}}^2
}$ is a square of a unit of $\mathfrak{o}_\fp$. When $\mathfrak{p}$
is inert or ramified in $K$, let $\mathfrak{P}$ be the place of
$K$ above $\mathfrak{p}$. Let $1, \omega$ be a basis of
$\mathfrak{o}_{K_\mathfrak{P}}$ above $\fo_\mathfrak{p}$. By
(\ref{eq:fund}) there exists a unit $u$ of $\fo$ such that
$d_{Q_\mathfrak{p}} = u n_{Q_\mathfrak{p}}^2 d_\omega$. On the other
hand, since $F(\sqrt{d_{Q\mathfrak{p}}})=F(\sqrt{d_\omega})$, we can
write $$\sqrt{d_{Q\mathfrak{p}}}=x+y\sqrt{d_\omega}\hskip 20pt
x,y\in F, \ y\neq 0.$$  From
$d_{Q\mathfrak{p}}=(x+y\sqrt{d_\omega})^2=x^2+y^2d_\omega +2xy
\sqrt{d_\omega} \in F$ we obtain $x=0$ and thus $u=(\frac{y}{n_{Q_{\mathfrak{p}}}})^{2}$. Now by
Lemma \ref{lem:crit:fund} (\ref{it:degree-2}),
$\frac{d_{Q_\mathfrak{p}}}{n^2_{Q_\mathfrak{p}}}$ is fundamental.
\end{proof}

\subsection{Equivalence classes of fundamental quadratic forms}

Let $K$ be a totally imaginary quadratic extension of $F$. We use
$\mathfrak{Q}_{K/F}$ to denote the set of fundamental definite
quadratic forms $(M,Q)$ over $\mathfrak{o}_{F}$ with $F(Q)=K$.

\begin{defn} Two quadratic forms $(M_1, Q_1)$ and $(M_2, Q_2)$ are
said to be {\it weakly equivalent} if there exists an
$\mathfrak{o}$-module isomorphism $\phi: M_1\rightarrow M_2$ and an
element $u\in F^\times$ such that $Q_1= u \cdot Q_2\circ (
\phi\times \phi)$. We write $(M_1, Q_1)\sim (M_2, Q_2)$ or simply
$Q_1\sim Q_2$. If further $u=1$, then $(M_1, Q_1)$ and $(M_2, Q_2)$
are (strongly) equivalent.
\end{defn}

If $(M_1,Q_1)\sim (M_2,Q_2)$, then $F(Q_1)=F(Q_2)$; $Q_1$ is
fundamental if and only if so is $Q_2$; $Q_1$ is definite if and
only if so is $Q_2$. Hence, if $(M_1,Q_1)\in \mathfrak{D}_{K/F}$,
and if $(M_1,Q_1)\sim (M_2,Q_2)$, then $(M_2,Q_2)\in
\mathfrak{Q}_{K/F}$. Let $[\mathfrak{Q}_{K/F}]$ (resp. $[\mathfrak{Q}_{K/F}]_{0}$) be the quotient set
of $\mathfrak{Q}_{K/F}$ modulo $\sim$ (resp. strong equivalence).

To each  ideal $\mathfrak{A}$ of $\mathfrak{o}_K$ we associate a
quadratic form $(\mathfrak{A}, Q_\mathfrak{A})$ by
$$ Q_\mathfrak{A}(z,w)= \frac{\mathrm{Tr}_{K/F}(\bar{z}w)}{2}
.$$ Then $q_{Q_\mathfrak{A}}=N_{K/F}(\cdot)$ and
$\mathfrak{n}_{Q_\mathfrak{A}}=N_{K/F}(\mathfrak{A})$. Note that for
each $\mathfrak{A}$, $(Q_\mathfrak{A})_F$ is induced by the quadratic
form $  Q_{K/F}$ on $K$ defined by the same formula.


Let $\mathfrak{d}_{K/F}$ be the relative different of $K$ over $F$.

\begin{lem}\label{lem:relation-d} If $\mathfrak{A}$ is a fractional ideal of $K$, then
$\mathfrak{d}_{Q_{\mathfrak{A}}}=\mathfrak{d}_{K/F}
N_{K/F}(\mathfrak{A})^{2}=\mathfrak{d}_{K/F}\mathfrak{n}_{Q_\mathfrak{A}}^{2}$.
\end{lem}
\begin{proof} We write $Q$ for $Q_\mathfrak{A}$ and $d_{K_{\mathfrak{P}}/F_{\mathfrak{p}}}$ for a generator of $\mathfrak{d}_{K_{\mathfrak{P}}/F_{\mathfrak{p}}}$. We only need to show that $$(d_{Q_{\mathfrak{p}}})=\prod_{\mathfrak{P}/\mathfrak{p}}(d_{K_{\mathfrak{P}}/F_{\mathfrak{p}}})N_{K_{\mathfrak{P}}/F_{\mathfrak{p}}}(\mathfrak{A}\mathfrak{o}_{K_{\mathfrak{P}}})^{2}$$ for each prime ideal $\mathfrak{p}$ of $F$. $\mathfrak{A}\mathfrak{o}_{\mathfrak{p}}=\mathfrak{o}_{\mathfrak{p}}\alpha\oplus\beta\mathfrak{o}_{\mathfrak{p}}$ as $\mathfrak{o}_{\mathfrak{p}}$-modules.

If $\mathfrak{p}$ splits in $K$, then $d_{Q_{\mathfrak{p}}}=(\alpha\beta)^{2}=\alpha^{2}\beta^{2}=\prod_{\mathfrak{P}/\mathfrak{p}}N_{K_{\mathfrak{P}}/F_{\mathfrak{p}}}(\mathfrak{A}\mathfrak{o}_{K_{\mathfrak{P}}})^{2}$. If $\mathfrak{p}$ is either inert or ramified in $K$, then $\mathfrak{A}\mathfrak{o}_{K_{\mathfrak{P}}}=\mathfrak{A}\mathfrak{o}_{F_\mathfrak{p}}=\mathfrak{o}_{F_\mathfrak{p}}\alpha\oplus\mathfrak{o}_{F_\mathfrak{p}}\beta$. Let $\sigma$ be the generator of ${\rm Gal}(K/F)$. We have $$(d_{Q_{\mathfrak{p}}})=(\alpha\sigma(\beta)-\beta\sigma(\alpha))^{2}\mathfrak{o}_{F_{\mathfrak{p}}}=(d_{K_{\mathfrak{P}}/F_{\mathfrak{p}}})N_{K_{\mathfrak{P}}/F_{\mathfrak{p}}}(\mathfrak{A}\mathfrak{o}_{K_{\mathfrak{P}}})^{2}.$$ This proves the lemma.
\end{proof}

Recall that $\Cl_K$ is the group of ideal classes of $K$. For a fractional
ideal $\mathfrak{A}$ of $K$, we write $[\mathfrak{A}]$ for its
class. There exists a complex conjugate action on $\Cl_K$. Indeed,
if $[\mathfrak{A}]=[\mathfrak{B}]$, then
$[\bar{\mathfrak{A}}]=[\bar{\mathfrak{B}}]$.

\begin{prop} \label{lem:ideal-to-form}
\begin{enumerate}
\item \label{it:idea-form-a} $Q_\mathfrak{A}$ is positive definite.
\item \label{it:idea-form-b} $F(Q_\mathfrak{A})=K$.
\item \label{it:idea-form-e} $(\mathfrak{A},Q_\mathfrak{A})$ is
fundamental.
\item \label{it:idea-form-c} If $\mathfrak{A}$ and $\mathfrak{B}$ are two fractional ideals
of $K$ in the same ideal class, then
$$(\mathfrak{A},Q_{\mathfrak{A}})\sim (\mathfrak{B},
Q_{\mathfrak{B}}).$$
\item \label{it:idea-form-d} $(\mathfrak{A},Q_{\mathfrak{A}})\sim
(\bar{\mathfrak{A}}, Q_{\bar{\mathfrak{A}}})$.
\item \label{it:ideal-equivalent-class} If $(\mathfrak{A}, Q_\mathfrak{A})\sim (\mathfrak{B},
Q_\mathfrak{B})$, then either $[\mathfrak{A}]=[\mathfrak{B}]$ or
$[\mathfrak{A}]=[\bar{\mathfrak{B}}]$.
\end{enumerate}
\end{prop}
\begin{proof} Assertion (\ref{it:idea-form-a}) follows from Lemma \ref{lem:crit} and the
obvious fact that for each nonzero $m\in \mathfrak{A}$, $N_{K/F}(m)$
is totally positive.

Let $\sigma$ be the nontrivial element of $\mathrm{Gal}(K/F)$. We
take $m=1$ and take $n$ to be a nonzero element of $\mathfrak{A}$
such that $\sigma(n)=-n$. Then $m$ and $n$ are linearly independent
over $\mathfrak{o}$. We have $$2Q_{\mathfrak{A}}(1,n)=
N_{K/F}(1+n)-N_{K/F}(1)-N_{K/F}(n)= (1+n)(1-n)-1 +n^2=0.$$ Thus
$$Q_{\mathfrak{A}}(1,n)^2-Q_{\mathfrak{A}}(1,1)Q_{\mathfrak{A}}(n,n)=-N_{K/F}(1)N_{K/F}(n)=n^2$$
and
$$ F(Q_\mathfrak{A})=F(\sqrt{Q_{\mathfrak{A}}(1,n)^2-Q_{\mathfrak{A}}(1,1)Q_{\mathfrak{A}}(n,n)})=F(n)=K
.$$ This proves (\ref{it:idea-form-b}).

Assertion (\ref{it:idea-form-e}) follows from Proposition
\ref{prop:fund} and Lemma \ref{lem:relation-d}.

If $\mathfrak{B}=\mathfrak{A}\cdot (\alpha)$ with $\alpha\in
K^\times$, we take $$\phi: \mathfrak{A}\rightarrow \mathfrak{B},
m\mapsto \alpha m ,$$ and put $u=N_{K/F}(\alpha^{-1})$. Then
$Q_\mathfrak{A}=u\cdot Q_\mathfrak{B}\circ \phi$. This shows
(\ref{it:idea-form-c}).

Assertion (\ref{it:idea-form-d}) is obvious.

We prove (\ref{it:ideal-equivalent-class}). Let $\phi:
\mathfrak{A}\rightarrow \mathfrak{B}$ be an isomorphism of $\mathfrak{o}$-modules such that
$Q_\mathfrak{B}\circ \phi = u\cdot Q_\mathfrak{A}$ for some $u\in
F^\times$. Take two nonzero elements $\alpha, \ \beta\in
\mathfrak{B}$ such that $Q_\mathfrak{B}(\alpha, \beta)=0$. We may
even take $\alpha\in F^\times \cap \mathfrak{B}$. Then $\beta$ is purely imaginary. Put
$\gamma=\frac{\alpha}{\phi^{-1}(\alpha)}$. Let $\mathfrak{B}'$ be
the ideal $(\gamma)\cdot \mathfrak{A}$. Then
$\alpha=\gamma\phi^{-1}(\alpha)\in \mathfrak{B}'$. Let $\phi'$ be
the map
$$ \mathfrak{B}' \rightarrow \mathfrak{B} , \hskip 10pt  a' \mapsto \phi ( a' \gamma^{-1}).
$$ Then $Q_{\mathfrak{B}} = v \cdot Q_{\mathfrak{B}'} \circ (\phi'^{-1}\times\phi'^{-1}) $
for some $v\in F^\times$. As
$\phi'(\alpha)=\phi(\phi^{-1}(\alpha))=\alpha$, we have $v=1$. From
$$\mathrm{Tr}_{K/F}(\alpha\phi'^{-1}(\beta))=2Q_{\mathfrak{B}'}(\alpha,\phi'^{-1}(\beta)
)= 2Q_{\mathfrak{B}}(\alpha,
\beta)=0,$$ we see there exists $w\in
F^\times$ such that $\phi'^{-1}(\beta)=w\beta$. From
$$ N_{K/F}(\beta) = Q_{\mathfrak{B}}(\beta, \beta) =
Q_{\mathfrak{B}'}(\phi'^{-1}(\beta), \phi'^{-1}(\beta)) = w^2
N_{K/F}(\beta) $$ we obtain $w=\pm 1$. When $w=1$, $\phi'$ is the
identity and $\mathfrak{B}=\mathfrak{B}'$. When $w=-1$, $\phi'$ is
the complex conjugate and $\mathfrak{B}=\mathfrak{B}'$.
\end{proof}

\begin{prop}\label{lem:form-to-ideal} For each fundamental  definite quadratic form $(M,Q)$ with $F(Q)=K$, there exists
a fractional ideal $\mathfrak{A}$ of $K$ such that $(M,Q)\sim
(\mathfrak{A}, Q_{\mathfrak{A}})$.
\end{prop}
\begin{proof} By the theory of classification of
$\mathfrak{o}$-modules, there exists an ideal $\mathfrak{a}$ of
$\mathfrak{o}$, and two elements $\alpha,\beta\in M$ such that
$M=\mathfrak{o}\alpha\oplus \mathfrak{a}\beta$. Let $a, b, c$ be
such that
$$ q_Q( x\alpha+ y \beta ) = ax^2+bxy+cy^2 $$ for $ x\in \mathfrak{o} $ and $ y\in
\mathfrak{a} $. We choose a square root $\sqrt{b^{2}-4ac}$ of $b^{2}-4ac$, and put $$ \mathfrak{A}= \mathfrak{o}\cdot 2a \oplus \mathfrak{a}\cdot(
-b+\sqrt{b^2-4ac} ) . $$ By the following Lemma
\ref{lem:is-an-ideal}, $\mathfrak{A}$ is an ideal of
$K$. Let $\phi$ be the map
$$ M\rightarrow \mathfrak{A} , \hskip 10pt  \alpha \mapsto -2a, \ \beta\mapsto -b+\sqrt{b^2-4ac} .
$$ Then $4a\cdot Q=Q_\mathfrak{A}\circ \phi$.
\end{proof}

\begin{lem}\label{lem:is-an-ideal} The above $\mathfrak{A}$  is a fractional ideal
of $\mathfrak{o}_K$.
\end{lem}
\begin{proof}
We only need to show that for each prime $\mathfrak{p}$ of
$\mathfrak{o}$, $ \mathfrak{o}_{F_{\mathfrak{p}}}\cdot \mathfrak{A}
$ is a fractional ideal of $K_{\mathfrak{p}}$. When
$\mathfrak{p}$ is split in $K$, there is nothing to prove. So we
assume that $\mathfrak{p}$ is either inert or ramified in $K$. Let
$1$ and $\omega$ be a basis of $\mathfrak{o}_{K_\mathfrak{P}}$ over
$\mathfrak{o}_{F_\mathfrak{p}}$.

Write $\mathfrak{a}_\mathfrak{p}=(\tilde{a})$ with $\tilde{a}\in
\mathfrak{o}_\mathfrak{p}$. Then $\alpha$ and $\tilde{a}\beta$ is a
basis of $M_\mathfrak{p}$ over $\fo_\mathfrak{p}$. We have
$$ q_{Q_\mathfrak{p}}(x\alpha+y\tilde{a}\beta)= a x^2+ b\tilde{a} xy +
c\tilde{a}^2 y^2 . $$ Since $Q$ is fundamental, by Lemma
\ref{lem:crit:fund} (\ref{it:degree-2}) there exists $u\in
\fo_\mathfrak{p}^\times$ such that
$$ d_{Q_\mathfrak{p}} = \frac{\tilde{a}^2 (b^2-4ac)}{n_{Q_\mathfrak{p}}^2} = u^2
d_\omega.
$$ Replacing $u$ by $-u$ if necessary, we may assume
$$ \sqrt{d_\omega} =   \frac{\tilde{a}  \sqrt{b^2-4ac}}{un_{Q_\mathfrak{p}}} . $$

What we need to show is that \begin{equation}
\label{eq:check}\omega\cdot (\mathfrak{o}_\mathfrak{p}\cdot2a\oplus
\mathfrak{o}_\mathfrak{p} \tilde{a} ( -b+\sqrt{b^2-4ac})) \subset
\mathfrak{o}_\mathfrak{p}\cdot2a\oplus \mathfrak{o}_\mathfrak{p} \tilde{a}
( -b+\sqrt{b^2-4ac}).\end{equation} Let $X^2+eX+f$ be minimal
polynomial of $\omega$ over $\fo_\mathfrak{p}$. Then
$\omega=\frac{-e\pm \sqrt{d_\omega}}{2}$. Without loss of generality
we may assume $\omega=\frac{-e+\sqrt{d_\omega}}{2}$. Check that
$$ \omega \cdot 2a = \frac{b\tilde{a}-uen_{Q_\mathfrak{p}}}{2un_{Q_\mathfrak{p}}} \cdot 2a+ \frac{a}{un_{Q_\mathfrak{p}}}\cdot \tilde{a} ( -b+\sqrt{b^2-4ac})
$$ and
$$\omega \cdot \tilde{a} ( -b+\sqrt{b^2-4ac})) = -\frac{\tilde{a}^2 c }{un_{Q_\mathfrak{p}}} \cdot 2a
 - \frac{eun_{Q_\mathfrak{p}}+b\tilde{a}}{2un_{Q_\mathfrak{p}}} \cdot \tilde{a} ( -b+\sqrt{b^2-4ac})) .
 $$

 Note that $\frac{a}{n_{Q_\mathfrak{p}}}$,
 $\frac{\tilde{a}b}{n_{Q_\mathfrak{p}}}$ and
 $\frac{\tilde{a}^2c}{n_{Q_\mathfrak{p}}}$ are all in
 $\fo_\mathfrak{p}$. So $\frac{a}{un_{Q_\mathfrak{p}}},-\frac{\tilde{a}^2 c
 }{un_{Q_\mathfrak{p}}}\in\fo_\mathfrak{p}$, and
the minimal polynomial $X^2+\frac{\tilde{a}b}{un_{Q_\mathfrak{p}}}X+\frac{\tilde{a}^2ac}{u^2n^2_{Q_\mathfrak{p}}}$ of $\frac{-b\tilde{a}+
 un_{Q_\mathfrak{p}}\sqrt{d_\omega}}{2un_{Q_\mathfrak{p}}}$ over $\fo_\mathfrak{p}$ is integral. That is  $\frac{-b\tilde{a}+
 un_{Q_\mathfrak{p}}\sqrt{d_\omega}}{2un_{Q_\mathfrak{p}}}\in\mathfrak{o}_{\mathfrak{p}}$. Therefore,
$$ \frac{b\tilde{a}-uen_{Q_\mathfrak{p}}}{2un_{Q_\mathfrak{p}}} = -\frac{-b\tilde{a}+
 un_{Q_\mathfrak{p}}\sqrt{d_\omega}}{2un_{Q_\mathfrak{p}}}+\omega
 $$ lies in $\fo_\mathfrak{p}$. Then so is
 $\frac{eun_{Q_\mathfrak{p}}+b\tilde{a}}{2un_{Q_\mathfrak{p}}}=\frac{b\tilde{a}-uen_{Q_\mathfrak{p}}}{2un_{Q_\mathfrak{p}}}+e$.
We obtain (\ref{eq:check}), as desired.
\end{proof}

\begin{thm} The correspondence $\mathfrak{A}\mapsto (\mathfrak{A},
Q_\mathfrak{A})$ provides a bijection
$$\mathrm{Cl}_K/\{\text{complex conjugate}\} \rightarrow
[\mathfrak{Q}_{K/F}] . $$  \end{thm}
\begin{proof} This follows from Propositions \ref{lem:ideal-to-form} and
 \ref{lem:form-to-ideal}.
\end{proof}

\begin{cor} \label{cor:estimation for hk by quadratic forms}We have $$ \sharp [\mathfrak{Q}_{K/F}] \leq h_K  \leq 2 \cdot \sharp
[\mathfrak{Q}_{K/F}] .$$
\end{cor}

\subsection{An estimate on $h_K$}

Let $(M,Q)$ be a quadratic form over $\fo$. We say that $s\in F$
is representable by $Q$, if there exists $m\in M\otimes_{\mathfrak{o}} F$ such that
$s=q_Q(m)$.

\begin{lem}\label{lem:e-non-depend} Let $(M,Q)$ be a definite quadratic form. If both $s_1$ and $s_2\in
\fo$ $(s_1s_2\neq 0)$ are representable by $Q$, then for each
place $v$ of $F$ and each place $\omega$ of $F(Q)$ above $v$, we have
$$ (s_1, F(Q)_{\omega}/F_{v} ) =(s_2, F(Q)_{\omega}/F_{v}).$$
\end{lem}
\begin{proof}If $v$ is finite, then we just repeat the argument of \cite[P63 Th\'{e}or\`{e}me
8]{BS}. If $v: F\hookrightarrow\mathbb{R}$ is an infinite place, then since $Q$ is definite, both sides are equal to the sign of $Q_{v}$. 
\end{proof}

\begin{defn} Let $(M,Q)$ be a definite quadratic form over $\fo$. For each place $v$ of $F$ we put
$$ e_{v}(Q)= (s, F(Q)_{\omega}/F_{v}), $$ where $s$ is a nonzero element of $F$ that is representable by $Q$, and $\omega$ is a place above $v$. By Lemma \ref{lem:e-non-depend} , $e_{v}(Q)$ does not depend on the
choice of $s$.
\end{defn}

It is easy to see that if $Q$ is strongly equivalent to $Q^{\prime}$, then $e_{v}(Q)=e_{v}(Q^{\prime})$ for all $v$.

In the following, we fix a (totally) imaginary quadratic extension $K$ of $F$.

\begin{prop}\label{prop:represent} $s\in \fo$ is representable by a fundamental definite quadratic
form of the form $\frac{Q_\mathfrak{A}}{\gamma}$ with $\mathfrak{A}$
an ideal of $K$ and $\gamma\in F$ such that
$N_{K/F}(\mathfrak{A})=(\gamma)$ is principal, if and only if for
each $\mathfrak{p}$ of $\mathfrak{o}$ inert in $K$,
$v_{\mathfrak{p}}(s)$ is even.
\end{prop}
\begin{proof}
Note that $v_{\mathfrak{p}}(s)$ is even for any prime $\mathfrak{p}$
of $\mathfrak{o}$ inert in $K$, if and only if $(s)$ is the norm
of an ideal of $\mathfrak{o}_K$.

If $s$ is representable by a quadratic form of the form
$\frac{Q_\mathfrak{A}}{\gamma}$, then there exists some $z\in
\mathfrak{A}$ such that $  s =  \frac{N_{K/F}(z)}{\gamma}. $ So $
(s) = N_{K/F}((z)\mathfrak{A}^{-1})  $ is the norm of an integral
ideal.

Conversely, if $(s)=N_{K/F}\mathfrak{B}$ is the norm of an ideal
$\mathfrak{B}$ of $\mathfrak{o}_K$. We take $\mathfrak{A}$ such that
$\mathfrak{A}\mathfrak{B}=(z)$ is a principal ideal. Put
$\gamma=s^{-1}N_{K/F}(z)$. Then $N_{K/F}\mathfrak{A}=(\gamma)$ is
principal. Write $\mathfrak{A}$ in the form
$\mathfrak{A}=\mathfrak{o}\alpha\oplus \mathfrak{a}\beta$. Then
$z\in \mathfrak{A}$ and $s = \frac{Q_\mathfrak{A}(z)}{\gamma} $.
\end{proof}

\begin{cor} \label{cor:ep(Q) at unramified primes}Let $\frac{Q_\mathfrak{A}}{\gamma}$ be as in Proposition
\ref{prop:represent}. Then for each prime $\mathfrak{p}$ of $\fo$
that is not ramified in $K$, we have
$e_\mathfrak{p}(\frac{Q_\mathfrak{A}}{\gamma})=1$.
\end{cor}
\begin{proof} Write $Q$ for $\frac{Q_\mathfrak{A}}{\gamma}$. Let $s\in
\fo$ be a nonzero element that is representable by $Q$. Then $e_\mathfrak{p}(Q)=(s, K_{\mathfrak{P}}/F_{\mathfrak{p}})$ for a prime $\mathfrak{P}$ over $\mathfrak{p}$. If $\mathfrak{p}$ is split in $K$, then $e_{\mathfrak{p}}(Q)=1$. If $\mathfrak{p}$ is inert in $K$, then by Proposition \ref{prop:represent},  $e_\mathfrak{p}(\frac{Q_\mathfrak{A}}{\gamma})=(\pi_\mathfrak{p}, K_{\mathfrak{P}}/F_{\mathfrak{p}})^{2}=1$. 
\end{proof}

\begin{cor}\label{cor:estimate} Associated to each system $
(i_{v})_{v\mid\infty\cdot\mathfrak{d}_{K/F}}$ with $i_{v}=\pm 1 $ (for each $v$) and $$\prod_{v\mid\infty\cdot\mathfrak{d}_{K/F}}i_{v}=1,$$ there
exists a fundamental definite quadratic form
$Q=\frac{Q_\mathfrak{A}}{\gamma}$ as in Proposition
\ref{prop:represent}  such that $e_{v}(Q)=i_{v}$
for any $v\mid\infty\cdot\mathfrak{d}_{K/F}$. In particular, $\sharp[\mathfrak{Q}_{K/F}]_{0}\ge 2^{t+n-1}$.
\end{cor}
\begin{proof} 
	For each $\mathfrak{p}|\mathfrak{d}_{K/F}$ we take $e_\mathfrak{p}\in \mathfrak{o}_\mathfrak{p}^\times$ such
that $(e_\mathfrak{p}, K_{\mathfrak{p}}/F_{\mathfrak{p}} )=i_{\mathfrak{p}}$. We have the fact that for each $e\in \mathfrak{o}$, if $e\equiv
e_\mathfrak{p} (\mathrm{mod} \ \mathfrak{d}_{K/F,\mathfrak{p}})$,
then $(e, K_{\mathfrak{p}}/F_{\mathfrak{p}})=(e_\mathfrak{p}, K_{\mathfrak{p}}/F_{\mathfrak{p}})$. By Dirichlet density theorem, there exists a principal prime $\mathfrak{q}=(e_{0})$ of $\mathfrak{o}$ with the property that $e_{0}\equiv
e_\mathfrak{p} (\mathrm{mod} \ \mathfrak{d}_{K/F,\mathfrak{p}})$ for every $\mathfrak{p}|\mathfrak{d}_{K/F}$ and the signs of $e_{0}$ at all infinite places are exactly given by $(i_{v})_{v\mid\infty}$. Then for
each prime $\mathfrak{p}'$ of $\fo$ such that $\mathfrak{p}'\nmid
\mathfrak{q}\mathfrak{d}_{K/F}$, we have
$(e_{0},K_{\mathfrak{p}^{\prime}}/F_{\mathfrak{p}^{\prime}})=1$. By Corollary \ref{cor:ep(Q) at unramified primes} as well as the reciprocity law, we have
$$1= (e_{0},K_{\mathfrak{q}}/F_{\mathfrak{q}})\times\prod_{v\mid\infty\cdot\mathfrak{d}_{K/F}}i_{v}=(e_{0},K_{\mathfrak{q}}/F_{\mathfrak{q}}).$$
This means
$\mathfrak{q}$ is split in $K$.

Now by Proposition \ref{prop:represent} there exists a fundamental
definite quadratic form $Q$ as desired such that $e_{0}$ is
representable by $Q$. We have $e_{v}(Q)=(e_{0}, K_{\omega}/F_{v})=i_{v}$ for each
$v\mid\infty\cdot\mathfrak{d}_{K/F}$.
\end{proof}

\begin{thm} \label{cor:estimate h by t}If $t$ is the number of different prime divisors of
$\mathfrak{d}_{K/F}$, then
$$ h_K \geq \frac{2^{t+n-1}}{[\mathfrak{o}^\times:(\mathfrak{o}^\times)^2]}.
$$ \end{thm}
\begin{proof} By Corollary \ref{cor:estimate}, we have $\sharp[\mathfrak{Q}_{K/F}]_{0}\ge 2^{t+n-1}.$ Notice that $(u, Q)\mapsto
	u\cdot Q$ gives a well defined group action $(\fo^\times/(\fo^\times)^2)\times[\mathfrak{Q}_{K/F}]_{0}\rightarrow [\mathfrak{Q}_{K/F}]_{0}.$ The natural map $[\mathfrak{Q}_{K/F}]_{0}\rightarrow[\mathfrak{Q}_{K/F}]$ extends to an injection $[\mathfrak{Q}_{K/F}]_{0}/(\fo^\times/(\fo^\times)^2)\rightarrow[\mathfrak{Q}_{K/F}]$. Then by Corollary \ref{cor:estimation for hk by quadratic forms}, we have 
	$$h_{K}\ge\sharp [\mathfrak{Q}_{K/F}]\ge\frac{\sharp[\mathfrak{Q}_{K/F}]_{0}}{[\mathfrak{o}^{\times}:(\mathfrak{o}^{\times})^{2}]}\ge\frac{2^{t+n-1}}{[\mathfrak{o}^\times:(\mathfrak{o}^\times)^2]}.$$
	This proves the theorem.
\end{proof}

\subsection{An estimate on $h$}\label{An estimate on $h$}
First we have the following observation:
\begin{lem}\label{lem: non eq units}
	Let $u_{1},\cdots,u_{n-1}\in\mathfrak{o}_{F}^{\times}$ be a basis of the free multiplicative group $\mathfrak{o}_{F}^{\times}/\{\pm1\}$. Then the set of all totally imaginary quadratic extensions $K$ over $F$ with $\mathfrak{o}_{K}^{\times}\ne\mathfrak{o}_{F}^{\times}$ is $$\{F(\sqrt[3]{-1})\}\cup \bigg{\{}F(\sqrt{v}):\ v=\pm\prod_{i=1}^{n-1}u_{i}^{\epsilon_{i}}\ \text{is totally negative},\epsilon_{i}\in\{0,1\}\ (i=1,\cdots,n-1)\bigg{\}}.$$ 
\end{lem}
\begin{proof}
	Let $u\in\mathfrak{o}_{K}^{\times}\setminus\mathfrak{o}_{F}^{\times}$. Then $K=F(u)$. Since both $\mathfrak{o}_{K}^{\times}$ and $\mathfrak{o}_{F}^{\times}$ have rank $n-1$, there exists $m>0$ such that $u^{m}\in F$. Therefore the minimal polynomial $P(X)$ of $u$ over $F$ is of degree 2, and divides $X^{m}-u^{m}$. Then we may write another root of $P(X)$ by $u^{\prime}=\zeta_{m^{\prime}}\cdot u$ for some $m^{\prime}$-th primitive root of unity ($m^{\prime}$ divides $m$). In particular, $\zeta_{m^{\prime}}=u^{\prime}/u\in K$ and then $2\ge[F(\zeta_{m^{\prime}}):F]=[\mathbb{Q}(\zeta_{m^{\prime}}):\mathbb{Q}]=\varphi(m^{\prime})$. Here $\varphi$ denotes the Euler function. We deduce that $m^{\prime}\in\{2,3,4,6\}$. 
	
	If $m^{\prime}=3$ or $6$, we have $K=F(\sqrt[3]{-1})$. In the case $m^{\prime}=2$ or $4$, $u^{\prime}=-u$, we have $u^{2}\in\mathfrak{o}_{F}^{\times}$, and then $K$ is of the form $$F(\sqrt{v}),\ v=\pm\prod_{i=1}^{n-1}u_{i}^{\epsilon_{i}}\ \text{is totally negative},\epsilon_{i}\in\{0,1\}\ (i=1,\cdots,n-1).$$
	This proves the lemma.
\end{proof}

\begin{rem}\label{rem:extensions by units}
By Lemma \ref{lem: non eq units}, the totally imaginary quadratic extensions $K$ over $F$ with $\mathfrak{o}_{K}^{\times}\ne\mathfrak{o}_{F}^{\times}$ form a finite set, and can be effectively computed. Then in the rest of our paper, we make the assumption that $ \mathfrak{o}_K^\times=\mathfrak{o}_F^\times $, and sometimes
write it by $\mathfrak{o}^{\times}$.	
\end{rem}

Let us consider the natural homomorphism
$$ \phi: \mathrm{Cl}_F \rightarrow \mathrm{Cl}_K,\ [\mathfrak{a}]\mapsto[\mathfrak{ao}_{K}] . $$  Put
$h=\sharp(\mathrm{Cl}_K/\mathrm{im}(\phi))$ and $h^{\prime}=\sharp(\mathrm{Cl}_{F}/{\rm ker}(\phi))$, and then $hh^{\prime}=h_{K}$. We take a set of fractional
ideals $\mathfrak{N}_1, \cdots \mathfrak{N}_{h}$ (resp. $\mathfrak{a}_{1},\cdots,\mathfrak{a}_{h^{\prime}}$) of $K$ (resp. $F$) that
represent $\mathrm{Cl}_K/\mathrm{im}(\phi)$ (resp. $\mathrm{Cl}_{F}/{\rm ker}(\phi)$).

For each $i$ ($1\leq i\leq h$), let $\mathscr{L}_i$ be the set of invertible
$\mathfrak{o}$-submodules $\mathcal{L}$ of $\mathfrak{N}_i$ defined
by
$$ \mathscr{L}_i = \{\mathcal{L}\subset \mathfrak{N}_i: \mathcal{L}=\mathfrak{a}_{j}\alpha\ \text{for some}\ j\in\{1,...,h^{\prime}\},\alpha\in K^{\ast}\} . $$
 We also put
$\mathscr{L}_{i}^{sa}=\{\mathcal{L}\in\mathscr{L}_{i}:\ \mathcal{L}\
\text{is saturated}\}$.

\begin{lem}\label{lem:ideal} Let $\mathfrak{M}$ be an ideal of $\mathfrak{o}_K$.
    \begin{enumerate}
    \item\label{it:ideal-1} There exists a unique $k\in\{1,\cdots,h\}$ and a unique $\mathcal{L}^{\prime}\in\mathscr{L}_{k}$ such that $\mathfrak{M}=\mathcal{L}^{\prime}\mathfrak{N}_{k}^{-1}$.
    \item\label{it:ideal-2} There exists a unique $i\in\{1,\cdots,
    h\}$, a unique $\mathcal{L}\in \mathscr{L}_i^{sa}$ and a unique ideal $\mathfrak{a}$ of $\mathfrak{o}_{F}$ such that $\mathfrak{M}=\mathfrak{a}\mathcal{L}\mathfrak{N}_{i}^{-1}$. \end{enumerate}
\end{lem}

\begin{proof} We first prove (\ref{it:ideal-1}). Let $[\mathfrak{M}]$ be the class of $\mathfrak{M}$ in
$\mathrm{Cl}_K$. Then there exists a $k$ such that
$[\mathfrak{M}\mathfrak{N}_k]$ lies in the image of $\phi$. Namely
there exists a $j\in\{1,...,h^{\prime}\}$ such that
$\phi([\mathfrak{a}_{j}])= [\mathfrak{M}\mathfrak{N}_k]$. This means
that $\mathfrak{M}=\alpha \mathfrak{a}_{j}  \mathfrak{o}_K
\mathfrak{N}_k^{-1}$ for some $\alpha\in K^{\ast}$. Since
$\mathfrak{M}\subset \mathfrak{o}_{K}$, we have
$\mathcal{L}^{\prime}=\mathfrak{a}_{j}\alpha\in\mathscr{L}_{k}$.
This shows the existence of $k$ and $\mathcal{L}^{\prime}$. To prove
the uniqueness, let $m$ and
$\tilde{\mathcal{L}}^{\prime}=\mathfrak{a}_{\ell}\beta\in\mathscr{L}_{m}$
be such that
$\mathfrak{M}=\tilde{\mathcal{L}}^{\prime}\mathfrak{N}_{m}^{-1}$.
Then $ \mathcal{L}^{\prime}\mathfrak{N}_{k}^{-1}
=\tilde{\mathcal{L}}^{\prime}\mathfrak{N}_{m}^{-1}$. By the
definition of $\mathfrak{N}_{1},...,\mathfrak{N}_{h}$ we have $k=m$.
Hence
$\mathfrak{a}_{j}\alpha\mathfrak{o}_{K}=\mathfrak{a}_{\ell}\beta\mathfrak{o}_{K}$.
By the definition of
$\mathfrak{a}_{1},...,\mathfrak{a}_{h^{\prime}}$ we have $j=\ell$.
Therefore, $\alpha$ differs from $\beta$ by a unit
$u\in\mathfrak{o}_{K}^{\times}$. The condition
$\mathfrak{o}_{K}^{\times}=\mathfrak{o}^{\times}$ implies that
$\mathcal{L}^{\prime}=\mathfrak{a}_{j}\alpha=\mathfrak{a}_{j}\beta=\tilde{\mathcal{L}}^{\prime}$.
This proves (\ref{it:ideal-1}).

We write $\mathfrak{M}=\mathcal{L}^{\prime}\mathfrak{N}_{k}^{-1}$
 as in (\ref{it:ideal-1}) and choose a saturated invertible submodule $\mathcal{L}^{\prime\prime}$ of $\mathfrak{N}_{k}$ containing $\mathcal{L}$.
 Then $\mathcal{L}^{\prime}=\mathfrak{a}\mathcal{L}^{\prime\prime}$ for some ideal $\mathfrak{a}$ of $\mathfrak{o}_{F}$. There exists $i\in\{1,...,h\}$ and $\mathcal{L}\in\mathscr{L}_{i}$ such that
 $\mathcal{L}\mathfrak{N}_{i}^{-1}=\mathcal{L}^{\prime\prime}\mathfrak{N}_{k}^{-1}$ (in fact, $i=k$). Clearly $\mathcal{L}\in\mathscr{L}_{i}^{sa}$, and
  $\mathfrak{M}=\mathfrak{a}\mathcal{L}\mathfrak{N}_{i}^{-1}$. This
  proves the existence part of (\ref{it:ideal-2}).
For the uniqueness, we assume that
$\mathfrak{a}\mathcal{L}\mathfrak{N}_{i}^{-1}=\mathfrak{b}\tilde{\mathcal{L}}\mathfrak{N}_{r}^{-1}$ with $\mathcal{L}\in\mathscr{L}_{i}^{sa}$ and $\tilde{\mathcal{L}}\in\mathscr{L}_{r}^{sa}$. So $i=r$.
  Write $\mathfrak{a}^{-1}\mathfrak{b}=\mathfrak{a}^{\prime-1}\mathfrak{b}^{\prime}$
  ($\mathfrak{a}^{\prime},\mathfrak{b}^{\prime}\subseteq\mathfrak{o}$, $\mathfrak{a}^{\prime}+\mathfrak{b}^{\prime}=\mathfrak{o}$).
 Then $\mathfrak{a}^{\prime}\mathfrak{o}_{K}$ divides
 $\mathfrak{b}^{\prime}\tilde{\mathcal{L}}\mathfrak{N}_{i}^{-1}$,
 and so divides $\tilde{\mathcal{L}}\mathfrak{N}_{i}^{-1}$. Since $\tilde{\mathcal{L}}$ is saturated, we have $\mathfrak{a}^{\prime}=\mathfrak{o}$.
  Similarly, $\mathfrak{b}^{\prime}=\mathfrak{o}$. Therefore $\mathfrak{a}=\mathfrak{b}$, and we have $\mathcal{L}\mathfrak{o}_{K}=\tilde{\mathcal{L}}\mathfrak{o}_{K}$. By (\ref{it:ideal-1}), $\mathcal{L}=\tilde{\mathcal{L}}$. This proves (\ref{it:ideal-2}).
\end{proof}

Let $\sigma$ be the nontrivial element of $\mathrm{Gal}(K/F)$. For
each $i=1,\cdots, h$, we consider the quadratic form
$$ Q_i : \mathfrak{N}_i\times \mathfrak{N}_i \rightarrow F, \hskip 10pt (x,y)\mapsto
\frac{x\sigma(y)+y\sigma(x)}{2}.
$$ By Lemma \ref{lem:relation-d}, its discriminant is $\mathfrak{d}_{K/F}N_{K/F}(\mathfrak{N}_i)^2$. As $K/F$ is totally
imaginary, $Q_i$ is positive definite.

Let $\zeta_F$ and $\zeta_K$ be the zeta functions defined by
$$ \zeta_F(s) = \sum_{\mathfrak{a}\subset\mathfrak{o}}
|\mathfrak{a}|_{\mathbb{R}}^{-s}
$$ and $$ \zeta_K(s) = \sum_{\mathfrak{M}\subset \mathfrak{o}_K}
|N_{K/\mathbb{Q}}(\mathfrak{M})|^{-s}$$ as usual, where $\mathfrak{a}$ and
$\mathfrak{m}$ run over the set of ideals of $\mathfrak{o}$ and
$\mathfrak{o}_K$ respectively. Obviously
$\frac{\zeta_K(s)}{\zeta_F(2s)}$ is a series with positive terms and
converges when $\mathrm{Re}(s)>1$.

\begin{prop} \label{prop:estimation of h} We have the following:
    \begin{enumerate}
        \item\label{it:estimation-1}For each $i\in\{1,...,h\}$, we fix a saturated invertible $\mathfrak{o}$-submodule $\mathcal{L}_{i}$ of $\mathfrak{N}_{i}$. Then we have
        \begin{eqnarray*}
        \zeta_{K}(s)&=&\zeta_{F}(2s)\sum_{i=1}^{h}\sum_{\mathcal{L}\in\mathscr{L}_{i}^{sa}}\bigg{|}N_{K/\mathbb{Q}}\bigg{(}\frac{\mathcal{L}\mathfrak{o}_{K}}{\mathfrak{N}_{i}}\bigg{)}\bigg{|}^{-s}\\
        &=&\zeta_{F}(2s)\sum_{i=1}^{h}\bigg{|}N_{K/\mathbb{Q}}\bigg{(}\frac{\mathcal{L}_{i}\mathfrak{o}_{K}}{\mathfrak{N}_{i}}\bigg{)}\bigg{|}^{-s}+\sum_{i=1}^{h}\sum_{\mathcal{L}\in\mathscr{L}_{i},\mathcal{L}\cap\mathcal{L}_{i}=0}\bigg{|}N_{K/\mathbb{Q}}\bigg{(}\frac{\mathcal{L}\mathfrak{o}_{K}}{\mathfrak{N}_{i}}\bigg{)}\bigg{|}^{-s}.\end{eqnarray*}
    \item\label{it:estimation-2} Let $\mathcal{L}_{i}$ be as in $\rm(a)$. Then
    $$ \sum_{i=1}^{h}\bigg{|}N_{K/\mathbb{Q}}\bigg{(}\frac{\mathcal{L}_{i}\mathfrak{o}_{K}}{\mathfrak{N}_{i}}\bigg{)}\bigg{|}^{-s}\ll \frac{\zeta_{K}(s)}{\zeta_{F}(2s)}. $$
        \item\label{it:estimation-3} If we write $\frac{\zeta_K(s)}{\zeta_F(2s)}$ in the form $\sum\limits_{n=1}^{+\infty}v_n
        n^{-s}$, then
        $$ \sum_{n<\frac{|\mathfrak{d}_{K/F}|^{1/2}}{2^{[F:\mathbb{Q}]}}} v_n \leq h. $$
    \end{enumerate}

\end{prop}
\begin{proof}
By Lemma \ref{lem:ideal}, $\zeta_{K}(s)$ can be writen in two forms
    \begin{eqnarray*}
        \zeta_{K}(s)&=&\zeta_{F}(2s)\sum_{i=1}^{h}\sum_{\mathcal{L}\in\mathscr{L}_{i}^{sa}}\bigg{|}N_{K/\mathbb{Q}}\bigg{(}\frac{\mathcal{L}\mathfrak{o}_{K}}{\mathfrak{N}_{i}}\bigg{)}\bigg{|}^{-s}\\
        &=&\sum_{i=1}^{h}\sum_{\mathcal{L}\in\mathscr{L}_{i},\mathcal{L}\subseteq\mathcal{L}_{i}}\bigg{|}N_{K/\mathbb{Q}}\bigg{(}\frac{\mathcal{L}\mathfrak{o}_{K}}{\mathfrak{N}_{i}}\bigg{)}\bigg{|}^{-s}+\sum_{i=1}^{h}\sum_{\mathcal{L}\in\mathscr{L}_{i},\mathcal{L}\cap\mathcal{L}_{i}=0}\bigg{|}N_{K/\mathbb{Q}}\bigg{(}\frac{\mathcal{L}\mathfrak{o}_{K}}{\mathfrak{N}_{i}}\bigg{)}\bigg{|}^{-s}.
    \end{eqnarray*}
For each $i$ and each $\mathcal{L}\in\mathscr{L}_{i}$ contained in $\mathcal{L}_{i}$, there exists $\mathfrak{a}\subseteq\mathfrak{o}$ such that $\mathfrak{a}\mathcal{L}_{i}=\mathcal{L}$. That is $\mathcal{L}\mathfrak{N}_{i}^{-1}=\mathfrak{a}\mathcal{L}_{i}\mathfrak{N}_{i}^{-1}$. On the other hand, if $\mathfrak{a}\mathcal{L}_{i}\mathfrak{N}_{i}^{-1}=\mathfrak{b}\mathcal{L}_{j}\mathfrak{N}_{j}^{-1}$ for $\mathfrak{a},\mathfrak{b}\subseteq\mathfrak{o}$, then $i=j$, and hence $\mathfrak{a}=\mathfrak{b}$.
This proves (\ref{it:estimation-1}).

By (\ref{it:estimation-1}), we have
$$ \frac{\zeta_K(s)}{\zeta_{F}(2s)} =\sum_{i=1}^{h}\sum_{\mathcal{L}\in\mathscr{L}_{i}^{sa}}\bigg{|}N_{K/\mathbb{Q}}\bigg{(}\frac{\mathcal{L}\mathfrak{o}_{K}}{\mathfrak{N}_{i}}\bigg{)}\bigg{|}^{-s}.$$ 
For each $i\in\{1,\cdots,h\}$, $\mathcal{L}_{i}$ is saturated. Then by Lemma \ref{lem:ideal}, there exists  a unique $\mathcal{L}^{\prime}_{i}\in\mathscr{L}_{i}^{sa}$ such that $\mathcal{L}_{i}\mathfrak{N}_{i}^{-1}=\mathcal{L}_{i}^{\prime}\mathfrak{N}_{i}^{-1}$ i.e. $|N_{K/\mathbb{Q}}(\mathcal{L}_{i}\mathfrak{N}_{i}^{-1})|^{-s}$ occures in the expansion of $\frac{\zeta_{K}(s)}{\zeta_{F}(2s)}$. This proves (\ref{it:estimation-2}).

By Lemma \ref{thm:only-one}, for each $i\in\{1,\cdots,h\}$ there exists at most one member  $\mathcal{L}$ of $\mathscr{L}_i^{sa}$ such
that
$$ |N_{K/\mathbb{Q}}(\mathcal{L} )| = I(Q_i, L) \leq \frac{|\mathfrak{d}_{K/F}|_\mathbb{R}^{1/2} |N_{K/\mathbb{Q}}(\mathfrak{N}_i)|}{2^{[F:\mathbb{Q}]}},
$$ or equivalently
$$ \left|N_{ K/\mathbb{Q} } ( \mathcal{L} \mathfrak{N}_i^{-1}) \right| \leq \frac{|\mathfrak{d}_{K/F}|_{\mathbb{R}}^{1/2} }{2^{[F:\mathbb{Q}]}}
,$$ which yields  (\ref{it:estimation-3}).
\end{proof}

\section{Hilbert modular form and Dirichlet series}
\label{sec:Gross-Zagier}

\subsection{Hilbert modular forms}\label{subsection 2.1}

We collect some facts on Hilbert modular forms. The main references
are \cite{TL 93} and \cite{Shi78}.

Let $\mathbb{A}_{F}$ be the ring of adeles of $F$. We write
$\mathfrak{d}$ for the absolute different of $F$. For each integral
ideal $\mathfrak{a}$ of $F$, let $W(\mathfrak{a})$ denote the
congruence subgroup
$$\Big{(}\mathrm{GL}_{2}(\mathbb{A}_{F,\infty})^{+}\times
\prod_{\mathfrak{p}}W_{\mathfrak{p}}(\mathfrak{a})\Big{)}\cap
\mathrm{GL}_{2}(\mathbb{A}_{F})$$ of
$\mathrm{GL}_{2}(\mathbb{A}_{F})$. Here, for each finite place
$\mathfrak{p}$ of $F$,
$$W_{\mathfrak{p}}(\mathfrak{a}):=\bigg{\{}
x=
\begin{pmatrix}
    a &b \\ c & d
\end{pmatrix}
\in
\begin{pmatrix}
    \mathfrak{o}_{\mathfrak{p}} &\mathfrak{d}_{\mathfrak{p}}^{-1} \\ (\mathfrak{ad})_{\mathfrak{p}} & \mathfrak{o}_{\mathfrak{p}}
\end{pmatrix}
\bigg{|}
a\mathfrak{o}_{\mathfrak{p}}+\mathfrak{a}\mathfrak{o}_{\mathfrak{p}}=\mathfrak{o}_{\mathfrak{p}},{\rm
det}(x)\in\mathfrak{o}_{\mathfrak{p}}^{\times}\bigg{\}}.   $$
Let $\{t_{1},\cdots,t_{h_{F}}\}$ be a set of finite ideles such that $t_{1}\mathfrak{o},\cdots,t_{h_{F}}\mathfrak{o}$ represent $Cl_{F}$. For each $\lambda\in\{1,\cdots,h_{F}\}$, we put $x_{\lambda}:={\rm diag}(1,t_{\lambda})\in GL_{2}(\mathbb{A}_{F})$. Then $$GL_{2}(\mathbb{A}_{F})=\coprod_{\lambda=1}^{h_{F}}GL_{2}(F)x_{\lambda}W(\mathfrak{a}).$$ Write $\varGamma_{\lambda}(\mathfrak{a}):=x_{\lambda}W(\mathfrak{a})x_{\lambda}^{-1}\cap GL_{2}(F)$. 

Let $\mathbf{f}$ be a normalized Hilbert newform of parallel weight
2 with respect to $W(\mathfrak{a})$, i.e. of level $\mathfrak{a}$. We write $\mathcal{H}$ for the usual Poincar\'{e} upper plane. Then $\mathbf{f}$ is uniquely determined by a tuple 
$(f_{1},\cdots,f_{h_{F}})$ of functions;
Here for each $\lambda\in\{1,\cdots,h_{F}\}$, $f_{\lambda}$ is a classical Hilbert modular form (defined on $\mathcal{H}^{n}$) with respect to $\varGamma_{\lambda}(\mathfrak{a})$.

Let $\lambda(\mathfrak{p},\mathbf{f})$ be Fourier coefficients of
$\mathbf{f}$. Then all 
$\lambda(\mathfrak{p},\mathbf{f})$ are real numbers
\cite[Propositions 2.5 and 2.8]{Shi78}.  When
$\mathfrak{p}\parallel\mathfrak{a}$,
$\lambda(\mathfrak{p},\mathbf{f})=\pm1$. When
$\mathfrak{p}^{2}\mid\mathfrak{a}$,
$\lambda(\mathfrak{p},\mathbf{f})=0$. By \cite[Theorem 1]{BD06}, when
$\mathfrak{p}\nmid\mathfrak{a}$,
$|\lambda(\mathfrak{p},\mathbf{f})|\le2\sqrt{|\mathfrak{p}|_{\mathbb{R}}}$.
One can attach to $\mathbf{f}$ a Dirichlet series $D(s,\mathbf{f})$,
which has an Euler product
$$  D(s,\mathbf{f})=\prod_{\mathfrak{p}\mid\mathfrak{a}}[1-\lambda(\mathfrak{p},\mathbf{f})
|\mathfrak{p}|_{\mathbb{R}}^{-s}]^{-1}\times
\prod_{\mathfrak{p}\nmid\mathfrak{a}}
[1-\lambda(\mathfrak{p},\mathbf{f})|\mathfrak{p}|_{\mathbb{R}}^{-s}+|\mathfrak{p}|_{\mathbb{R}}^{1-2s}]^{-1}.$$
The complete Dirichlet series
$$R(s,\mathbf{f}):=|\mathfrak{ad}^{2}|_{\mathbb{R}}^{\frac{s}{2}}(2\pi)^{-ns}\varGamma(s)^{n}D(s,\mathbf{f})$$
extends to an entire function and satisfies a functional equation
$$R(s,\mathbf{f})=\epsilon_{\mathbf{f}}\cdot
R(2-s,\mathbf{f}),\hskip 10pt \epsilon_{\mathbf{f}}\in\{1, -1\}.$$

Recall that $\mathbf{f}$ generates a cuspidal automorphic
representation $\pi_{\mathbf{f}}$ (with trivial central character)
of $\mathrm{GL}_{2}(\mathbb{A}_{F})$. The $L$-function
$L(s,\pi_{\mathbf{f}})$ attached to $\pi_{\mathbf{f}}$ satisfies
$$L\bigg{(}s-\frac{1}{2},\pi_{\mathbf{f}}\bigg{)}=|\mathfrak{ad}^{2}|_{\mathbb{R}}^{-\frac{s}{2}}R(s,\mathbf{f}).$$

Write $\vee^{2}\pi_{\mathbf{f}}$ for the second symmetric power
lifting (not necessarily cuspidal) of $\pi_{\mathbf{f}}$ to
$\mathrm{GL}_{3}(\mathbb{A}_{F})$; $\vee^{2}\pi_{\mathbf{f}}$ is
cuspidal if and only if $\pi_{\mathbf{f}}$ is not dihedral
\cite[Theorem 9.3 and Remark 9.9]{GJ 78}. Put
$$L_{2}(s,\pi_{\mathbf{f}}):=\frac{L(s,\pi_{\mathbf{f}}\times \pi_{\mathbf{f}})}{L_{F}(s,1)}.$$ It is always an entire function [loc. cit. Theorem 2.2 (4)]. When ${\rm
	Re}(s)\gg0$, we have
\begin{eqnarray*}   L(s,\vee^{2}\pi_{\mathbf{f}})&=& 2^{-n}\pi^{-\frac{3ns}{2}-2n}\varGamma\big{(}\frac{s+1}{2}\big{)}^{2n}\varGamma\big{(}\frac{s}{2}+1\big{)}^{n}\times\prod_{\mathfrak{p}\parallel\mathfrak{a}}[1-|\mathfrak{p}|_{\mathbb{R}}^{-1-s}]^{-1}\\
	&\times&\prod_{\mathfrak{p}\nmid\mathfrak{a}}[(1-\lambda(\mathfrak{p},\mathbf{f})|\mathfrak{p}|_{\mathbb{R}}^{-\frac{s+1}{2}}+|\mathfrak{p}|_{\mathbb{R}}^{-s})(1-|\mathfrak{p}|_{\mathbb{R}}^{-s})(1+\lambda(\mathfrak{p},\mathbf{f})|\mathfrak{p}|_{\mathbb{R}}^{-\frac{s+1}{2}}+|\mathfrak{p}|_{\mathbb{R}}^{-s})]^{-1}.
\end{eqnarray*}

If $\vee^{2}\pi_{\mathbf{f}}$ is cuspidal, then
$L(s,\vee^{2}\pi_{\mathbf{f}})=L_{2}(s,\pi_{\mathbf{f}})$, which does not vanish at $s=1$. 
We use $L_{\rm f}(s,\vee^{2}\pi_{\mathbf{f}})$ to denote the finite part of $L(s,\vee^{2}\pi_{\mathbf{f}})$. The value $L_{\rm f}(1,\vee^{2}\pi_{\mathbf{f}})$ can be computed as follows. As in \cite[2.3]{JO 84}, we have
$$L_{\rm f}(s-1,\vee^{2}\pi_{\mathbf{f}})=\zeta_{F}(2s-2)\prod_{\mathfrak{p}\mid\mathfrak{a}}(1+|\mathfrak{p}|_{\mathbb{R}}^{1-s})\times\frac{\sum_{\mathfrak{m}}\lambda(\mathfrak{m},\mathbf{f})^{2}|\mathfrak{m}|_{\mathbb{R}}^{-s}}{\zeta_{F}(s-1)}.$$
Note that $\sum_{\mathfrak{m}}\lambda(\mathfrak{m},\mathbf{f})^{2}|\mathfrak{m}|_{\mathbb{R}}^{-s}$ is nothing but the convolution Dirichlet series $D(s,\mathbf{f},\mathbf{f})$ defined in \cite[Section 4]{Shi78}. By [loc. cit. Proposition 4.9], it has a simple pole at $s=2$ with residue
$${\rm Res}_{s=2}D(s,\mathbf{f},\mathbf{f})= \frac{2^{n-2}\pi^{n}(4\pi)^{2n}R_{F}}{|d_{F}|^{3/2}\zeta_{F}(2)|\mathfrak{a}|_{\mathbb{R}}\prod_{\mathfrak{p}\mid\mathfrak{a}}(1+|\mathfrak{p}|_{\mathbb{R}}^{-1})}\sum_{\lambda=1}^{h_{F}}\int_{\varGamma_{\lambda}(\mathfrak{a})\backslash\mathcal{H}^{n}}f_{\lambda}\bar{f}_{\lambda}\prod_{\nu=1}^{n}{\rm d}x_{\nu}{\rm d}y_{\nu} .$$
Here $R_{F}$ denotes the regulator of $F$. Clearly $\zeta_{F}(s-1)$ has a simple pole at $s=2$ with residue $\frac{2^{n-1}h_{F}R_{F}}{\sqrt{|d_{F}|}}.$ Then finally we get $$L_{\rm f}(1,\vee^{2}\pi_{\mathbf{f}})=\frac{\pi^{n}(4\pi)^{2n}}{2h_{F}|d_{F}||\mathfrak{a}|_{\mathbb{R}}}\sum_{\lambda=1}^{h_{F}}\int_{\varGamma_{\lambda}(\mathfrak{a})\backslash\mathcal{H}^{n}}f_{\lambda}\bar{f}_{\lambda}\prod_{\nu=1}^{n}{\rm d}x_{\nu}{\rm d}y_{\nu}.$$

Let $\chi$ be a quadratic Hecke character (with conductor
$\mathfrak{d}_{\chi}$). We write $\mathbf{f}\otimes \chi$ for the
normalized Hilbert newform associated to $\pi_{\mathbf{f}}\otimes
\chi$ \cite[Theorems 1 and 3]{WC 73}, which is of parallel weight 2.
 We denote its level by $\mathfrak{a}_{\chi}$ ($\mathfrak{a}_{\chi}|\mathfrak{ad}^{2}_{\chi}$). Note that, in general, $\mathbf{f}\otimes \chi$ may be different to $\mathbf{f}_\chi$, the twist of $\mathbf{f}$ by $\chi$.
  Similarly, we have
$$R(s,\mathbf{f}\otimes\chi)=\epsilon_{\mathbf{f}\otimes\chi}\cdot
R(2-s,\mathbf{f}\otimes\chi),\hskip 10pt
\epsilon_{\mathbf{f}\otimes\chi}\in\{1,-1\}.$$

\begin{lem}\label{lem:epsilon-factor}
    Let $\mathbf{f}$ and $\chi$ be as above. We use $\chi^{\ast}$ $($resp. $\chi_{\rm f}$$)$ to denote the ideal character $($resp. Dirichlet character$)$ associated to $\chi$ \cite[Chapter 7, Section 6]{JN99}.
     Assume that $\mathfrak{a}$ is square free. If we write $\mathfrak{a}=\mathfrak{a}_{1}\mathfrak{a}_{2}$ $($${\rm g.c.d}(\mathfrak{a}_{1},\mathfrak{d}_{\chi})=\mathfrak{o}$
     and $\mathfrak{a}_{2}\mid\mathfrak{d}_{\chi}$$)$, then
    $$\epsilon_{\mathbf{f}\otimes\chi}=\chi_{\rm f}(-1)\chi^{\ast}(\mathfrak{a}_{1})\cdot\prod_{\mathfrak{p}\mid\mathfrak{a}_{2}}(-\lambda(\mathfrak{p}))\cdot\epsilon_{\mathbf{f}}.$$
\end{lem}
\begin{proof}
Since $\mathfrak{a}$ is square free, the level of $\mathbf{f}\otimes\chi$ is equal to
$\mathfrak{a}_{1}\mathfrak{d}_{\chi}^{2}$. Comparing the functional
equations of $L_{2}(s,\pi_{\mathbf{f}})$ and $R(s,\mathbf{f})$, we
have
$$\epsilon(s,\pi_{\mathbf{f}})=|\mathfrak{ad}^{2}|_{\mathbb{R}}^{\frac{1}{2}-s}\cdot\epsilon_{\mathbf{f}}\hskip10pt
\text{and} \hskip10pt\epsilon(s,\pi_{\mathbf{f}}\otimes\chi)
=|\mathfrak{\mathfrak{a}_{1}}\mathfrak{d}_{\chi}^{2}\mathfrak{d}^{2}|_{\mathbb{R}}^{\frac{1}{2}-s}\cdot\epsilon_{\mathbf{f}\otimes\chi}.$$
Note that both $\pi_{\mathbf{f}}$ and $\pi_{\mathbf{f}}\otimes\chi$
have no supercuspidal local components. Then our assertion follows
from simple calculations of automorphic
$\epsilon$-factors(\cite[Chapter 1, Section 15]{RG 70}).
\end{proof}

\subsection{Certain Dirichlet series}
In this subsection, let $K/F$ be as before. We consider two complex Dirichlet series
$\varPhi(s)$ and $\varPsi(s)$ satisfying the following conditions.

\begin{condition}\label{condition} (Compare to \cite[Section 3 C1-C4]{JO
84}.)
    \begin{enumerate}
        \item\label{condition-1} \label{it-Dirichlet-1}\label{condition-2} $\varPsi(s)<<\zeta_{F}(2s-1)^{2}$ and $\varPhi(s)<<\frac{\zeta_{K}(s-\frac{1}{2})^{2}}{\zeta_{F}(2s-1)^{2}}$.
        \item\label{condition-3} \label{it-Dirichlet-2} $\varPhi$ has an Euler product $\varPhi(s)=\prod_{\mathfrak{p}}\varPhi_{\mathfrak{p}}(s)$.
        Moreover, each $\varPhi_{\mathfrak{p}}(s)$ has the form
        $$\frac{(1-\alpha |\mathfrak{p}|_{\mathbb{R}}^{-s})(1-\beta |\mathfrak{p}|_{\mathbb{R}}^{-s})}{(1-\alpha^{\prime}
        |\mathfrak{p}|_{\mathbb{R}}^{-s})(1-\beta^{\prime} |\mathfrak{p}|_{\mathbb{R}}^{-s})} $$ where $|\alpha|,|\beta|,|\alpha^{\prime}|,|\beta^{\prime}|\le\sqrt{|\mathfrak{p}|_{\mathbb{R}}}$.
        \item\label{condition-4} \label{it-Dirichlet-3} There exists a constant $M>0$ such that the
        function $(|\mathfrak{d}_{K/F}|_{\mathbb{R}}M)^{s}\varGamma(s)^{2n}\varPhi(s)\varPsi(s)$ can be extended to an entire function
        which is rapidly decreasing in every vertical strip (when $s$ tends to infinity). Moreover, it is invariant under the change of variable $s\mapsto 2-s$
        and has a zero of order $\ge$ 3 at $s=1$.
        \item \label{it-Dirichlet-4} $\varPsi(s)$ is holomorphic on the half-plane $\{{\rm Re}(s)\ge1\}$ and has a simple zero at $s=1$. Moreover,
        the function $M^{s}\varGamma(s)^{2n}\varPsi(s)$ is rapidly decreasing in every vertical strip (when $s$ tends to infinity) of the half-plane $\{{\rm Re}(s)\ge1\}$.
    \end{enumerate}
\end{condition}

\begin{lem}\label{lem:hilbert modular form}
Let $\mathbf{f}$ be as in Subsection \ref*{subsection
2.1} and $\chi=(\hskip10pt, K/F)$. Suppose that the following conditions hold.
\begin{enumerate}
    \item\label{it-modular form-2} ${\rm ord}_{s=1}D(s,\mathbf{f})\ge3.$
        \item\label{it-modular form-3} $\epsilon_{\mathbf{f}\otimes\chi}=\epsilon_{\mathbf{f}}$.
    \end{enumerate}
Put
$$M=\frac{\sqrt{|\mathfrak{a}\mathfrak{a}_{\chi}\mathfrak{d}^{4}|_{\mathbb{R}}}}{(2\pi)^{2n}|\mathfrak{d}_{\chi}|_{\mathbb{R}}}.$$
For every finite place $\mathfrak{p}$ of $F$, we put
 \begin{eqnarray*}  \varPsi(s)_{\mathfrak{p}}:=
    \begin{cases}
        &(1-|\mathfrak{p}|_{\mathbb{R}}^{1-2s})\times L_{\mathfrak{p}}(2s-1,\vee^{2}\pi_{\mathbf{f}}), \hskip10pt \text{if } \mathfrak{p}\nmid\mathfrak{a}\\
        &L_{\mathfrak{p}}(2s-1,\vee^{2}\pi_{\mathbf{f}}),\hskip10pt \text{if } \mathfrak{p}\mid\mathfrak{a}\ \text{and}\ \mathfrak{p}^{2}\nmid{\rm g.c.d}(\mathfrak{a},\mathfrak{d}_{\chi}^{2})\\
        &D(\mathbf{f}\otimes\chi,s)_{\mathfrak{p}},\hskip 10pt  \text{if } \mathfrak{p}^{2}\mid{\rm g.c.d}(\mathfrak{a},\mathfrak{d}_{\chi}^{2})
    \end{cases}
 \end{eqnarray*}
 and define $$\varPsi(s):=\prod_{\mathfrak{p}}\varPsi_{\mathfrak{p}}(s) \hskip10pt \text{and}\hskip10pt \varPhi(s):=\frac{D(s,\mathbf{f})D(s,\mathbf{f}\otimes\chi)}{\varPsi(s)}.$$
 Then the triple $(\varPsi(s), \varPhi(s), M)$ satisfies Condition \ref{condition}.
\end{lem}
\begin{proof}
The proof is similar to the argument in \cite[Section 4.2]{JO 84}.
\end{proof}

\begin{example}\label{ex: modular form-1} (Gross-Zagier's example. See \cite[4.3]{JO
84}.) Let $f' $ be the elliptic newform of level 37 (of weight 2)
associated to the weil curve $y^{2}+y=x^{3}+x^{2}-23x-50.$ Note that
the fourier coefficient $\lambda(37,f')$ is equal to 1 i.e. the weil curve on the above has split multiplicative reduction at $p=37$. Put
$$f:=f'\otimes\chi_{0}\hskip10pt\text{where}\hskip10pt\chi_{0}=(\cdot
,\mathbb{Q}(\sqrt{-139})/\mathbb{Q}).$$ Then $f$ is an
 elliptic newform of level $37\times 139^{2}$, and ${\rm ord}_{s=1}L(f,s)=3$.

We claim that $\pi_{f}$ is not dihedral. Otherwise, there exists a nontrivial Hecke character $\eta$
(with conductor $d_{\eta}$) such that
$\pi_{f}\otimes\eta\cong\pi_{f}$. Then $\eta^{2}=1$ and
$f'\otimes\eta=f'$. The exact level of $f'\otimes\eta$ is ${\rm
l.c.m}(37,d_{\eta}^{2})$, which implies that $d_{\eta}=1$. Then the quadratic field associated to $\eta$ is unramified, and hence is equal to $\mathbb{Q}$ i.e. $\eta=1$. A contradiction. \end{example}

\begin{example}\label{ex: modular form-2}
We construct $\mathbf{f}$ that satisfies (\ref{it-modular form-2}) of Lemma \ref{lem:hilbert modular
form} in the case when  $F/\mathbb{Q}$ is solvable. The tool we use
is Arthur and Clozel's base change \cite{JA and LC}.

Since $F/\mathbb{Q}$ is solvable, we can find a chain of fields
extensions $$\mathbb{Q}=F_{1}\subset F_{2}\subset\cdots\subset
F_{r}=F$$ such that each $F_{i+1}/F_{i}$ is cyclic of prime order
$\ell_{i}$. We inductively construct a cuspidal representation
$\pi_{i}$ of $\mathrm{GL}_{2}(\mathbb{A}_{F_{i}})$ (with trivial
central character) for $(i=1,\cdots, r)$ that satisfies
$${\rm ord}_{s=1}L(s,\pi_{i})\ge3 , \hskip10pt   \pi_{i}\ \text{is not dihedral},$$
 and   satisfies the following condition $(*_\mathfrak{p})$ for all
but finitely many finite places $\mathfrak{p}$ of $F_{i}$:

$(*_\mathfrak{p})$  the local component $\pi_{i,\mathfrak{p}}$ is
unramified principal series
$\pi_{\mathfrak{p}}(\chi_{i,\mathfrak{p},1},\chi_{i,\mathfrak{p},2})$
such that
$$\lambda_{\pi_{i}}(\mathfrak{p}):=|N_{F_{i}/\mathbb{Q}}(\mathfrak{p})|^{1/2}(\chi_{i,\mathfrak{p},1}(\pi_{\mathfrak{p}})+\chi_{i,\mathfrak{p},2}(\pi_{\mathfrak{p}}))\
\text{is a rational integer}.$$

When $i=1$, put $\pi_{1}:=\pi_{f}$. We have already known that
$\pi_{1}$ is not dihedral and ${\rm ord}_{s=1}L(s,\pi_{1})=3$. When
$p\ne37,139$, $\pi_{1,\mathfrak{p}}$ is unramified principal series
and we have $\lambda_{\pi_{1}}(p)\in\mathbb{Z}$.

Assume that we have already constructed $\pi_i$. Put $\eta_{i}=(\cdot
,F_{i+1}/F_{i})$. Firstly, since $\pi_{i}$ is not dihedral, $\vee^{2}\pi_{i}$ is cuspidal (Subsection \ref{subsection 2.1}). We claim that
$\vee^{2}\pi_{i}\otimes\eta_{i}\ncong\vee^{2}\pi_{i}$. Indeed, we take a
finite place $\mathfrak{p}$ of $F_i$ such that $\pi_i$ satisfies
$(*_\mathfrak{p})$. We may choose $\mathfrak{p}$ that is inert
in $F_{i+1}$. Then $\eta_{i,\mathfrak{p}}(\pi_\mathfrak{p})\ne1$
where $\pi_\mathfrak{p}$ is a uniformizer of
$F_{i,\mathfrak{p}}$, and
 $$(\vee^{2}\pi_{i})_{\mathfrak{p}}\cong{\rm Ind}(\mathrm{GL}_{3}(F_{i,\mathfrak{p}}),B_{3}
 (F_{i,\mathfrak{p}}),\chi_{i,\mathfrak{p},1}\chi_{i,\mathfrak{p},2}^{-1},1,\chi_{i,\mathfrak{p},2}\chi_{i,\mathfrak{p},1}^{-1}).$$
The condition
$(\vee^{2}\pi_{i})_{\mathfrak{p}}\otimes\eta_{i,\mathfrak{p}}\cong(\vee^{2}\pi_{i})_{\mathfrak{p}}$
is equivalent (without loss of generality) to
$$\eta_{i,\mathfrak{p}}=\chi_{i,\mathfrak{p},1}\chi_{i,\mathfrak{p},2}^{-1}\hskip10pt\text{and}\hskip10pt\eta_{i,\mathfrak{p}}^{3}=1.$$
 Then by the relations $\lambda_{\pi_{i}}(\mathfrak{p})=|N_{F_{i}/\mathbb{Q}}(\mathfrak{p})|^{1/2}(\chi_{i,\mathfrak{p},1}(\pi_{\mathfrak{p}})+\chi_{i,\mathfrak{p},2}(\pi_{\mathfrak{p}}))$
  and $\chi_{i,\mathfrak{p},1}(\pi_{\mathfrak{p}})\chi_{i,\mathfrak{p},2}(\pi_{\mathfrak{p}})=1$, $\eta_{i,\mathfrak{p}}(\pi_{\mathfrak{p}})$
  satisfies a complex equation $$X^{2}+(2-\lambda_{\pi_{i}}(\mathfrak{p})^{2}|N_{F_{i}/\mathbb{Q}}(\mathfrak{p})|^{-1})X+1=0.$$
As $\eta_{i,\mathfrak{p}}(\pi_{\mathfrak{p}})^3=1$,
$\eta_{i,\mathfrak{p}}(\pi_{\mathfrak{p}})$ fits into another
equation $X^{2}+X+1=0$. To conclude, we have
$\lambda_{\pi_{i}}(\mathfrak{p})=\pm
|N_{F_{i}/\mathbb{Q}}(\mathfrak{p})|^{1/2}$, which contradicts to
$\lambda_{\pi_{i}}(\mathfrak{p})\in\mathbb{Z}$. Then claim holds.

 Since $\pi_{i}$ is dihedral, $\pi_{i}\otimes\eta_{i}\ncong\pi_{i}$. Then using base change lifting, we can lift $\pi_{i}$ to a cuspidal representation $\pi_{i+1}$ (with trivial central character) of $\mathrm{GL}_{2}(\mathbb{A}_{F_{i+1}})$. Moreover we have
 $$L(s,\pi_{i+1})=\prod_{j=0}^{\ell_{i}-1}L(s,\pi_{i}\otimes\eta_{i}^{j})$$
 (\cite[Ch.1 Proposition 6.9, Ch.3 Theorem 4.2 and Theorem 5.1]{JA and LC}).
 In particular,
 $${\rm ord}_{s=1}L(s,\pi_{i+1})\ge{\rm ord}_{s=1}L(s,\pi_{i})\ge3.$$
 
 By the claim in the above, using base change lifting again, $\vee^{2}\pi_{i}$ lifts to a cuspidal representation $\varPi_{i+1}$ of $\mathrm{GL}_{3}(\mathbb{A}_{F_{i+1}})$. Then we have $\varPi_{i+1}=\vee^{2}\pi_{i+1}$. In particular, $\vee^{2}\pi_{i+1}$ is cuspidal, which implies that $\pi_{i+1}$
 is not dihedral (Subsection \ref{subsection 2.1}).

 Next we show that, if $\pi_i$ satisfies $(*_\mathfrak{p})$, then
 $\pi_{i+1}$ satisfies $(*_\mathfrak{q})$ for each finite prime $\mathfrak{q}$ of $F_{i+1}$ above $\mathfrak{p}$.
Of course, the local components $\pi_{i+1,\mathfrak{q}}$ is an
unramified principal series. Put $\alpha=|N_{F_{i}/\mathbb{Q}}(\mathfrak{p})|^{1/2}
\chi_{i,\mathfrak{p},1}(\pi_{\mathfrak{p}})$ and
$\beta=|N_{F_{i}/\mathbb{Q}}(\mathfrak{p})|^{1/2}\chi_{i,\mathfrak{p},2}(\pi_{\mathfrak{p}})$.
Then
$$\lambda_{\pi_{i+1}}(\mathfrak{q})=
|N_{F_{i}/\mathbb{Q}}(\mathfrak{p})|^{f_\mathfrak{p}/2}(\chi_{i,\mathfrak{p},1}(\pi_{\mathfrak{p}})^{f_\mathfrak{p}}+\chi_{i,\mathfrak{p},2}(\pi_{\mathfrak{p}})^{f_\mathfrak{p}})
=\alpha^{f_\mathfrak{p}}+\beta^{f_\mathfrak{p}}.$$ Here,
  $f_\mathfrak{p}\in\{1,\ell_{i}\}$ is the inertial degree of $\mathfrak{q}$ over $\mathfrak{p}$.
As $\alpha+\beta$ and $\alpha\beta$ are in $\mathbb{Z}$, so is
  $\lambda_{\pi_{i+1}}(\mathfrak{q})=\alpha^{f_\mathfrak{p}}+\beta^{f_\mathfrak{p}}$.

By induction, $\pi_{1}$ lifts to a cuspidal representation $\pi_{r}$ of $GL_{2}(\mathbb{A}_{F})$, which  corresponds to a Hilbert newform $\mathbf{f}$ over $F$. Clearly, $\mathbf{f}$ satisfies (\ref{it-modular form-2}) of Lemma
\ref{lem:hilbert modular form}. This complete the construction.

\begin{rem}\label{rem: Hilbert modular form}
We can also lift $f'$ in the same way to $\mathbf{f}^{\prime}$ over
$F$. Then $\mathbf{f}=\mathbf{f}^{\prime}\otimes(\chi_{0}\circ
N_{\mathbb{A}_{F}/\mathbb{A}_{\mathbb{Q}}})$. The level of
$\mathbf{f}^{\prime}$ is equal to
$\prod_{\mathfrak{p}\mid37}\mathfrak{p}$, and for
$\mathfrak{p}\mid37$ we have
$\lambda(\mathfrak{p},\mathbf{f}^{\prime})=\lambda(37,f')^{f_\mathfrak{p}}=1$.
Here $f_\mathfrak{p}$ denotes the inertial degree of $\mathfrak{p}$. Note that when $F$ is fixed, the Fourier coefficients of $\mathbf{f}$ (resp. $\mathbf{f}^{\prime}$) can be effectively computed in terms of $f$ (resp. $f^{\prime}$).
\end{rem}

\end{example}

\begin{prop}\label{prop:epsilon factors}
Suppose that $F/\mathbb{Q}$ is solvable of degree $n$. Let
$\mathbf{f}$ be as in Example \ref{ex: modular form-2}, and put
$\chi:=(\hskip10pt,K/F)$. Write
$37\mathfrak{o}_{F}=\mathfrak{p}_{1}^{e}\cdots\mathfrak{p}_{s}^{e}$.
If $n+s$ is even, and if none of
$\mathfrak{p}_{1},\cdots\mathfrak{p}_{s}$ split in $K$, then
$(\mathbf{f},\chi)$ satisfies
$(\ref{it-modular form-2})$ and $(\ref{it-modular form-3})$ of Lemma
\ref{lem:hilbert modular form}. In particular, $\varPsi(s)$ and
$\varPhi(s)$ constructed in Lemma \ref{lem:hilbert modular form} by
using $(\mathbf{f},\chi)$ satisfy Condition \ref{condition}.
\end{prop}

\begin{proof} We have already known that $\mathbf{f}$ satisfies (\ref{it-modular form-2}) of Lemma \ref{lem:hilbert
modular form}. We check that $(\mathbf{f},\chi)$ also satisfies
(\ref{it-modular form-3}) of Lemma \ref{lem:hilbert modular form}.

    Suppose that none of $\mathfrak{p}_{1},\cdots,\mathfrak{p}_{s}$ split in $K$, then $\chi^{\ast}(\mathfrak{p}_{i})=-1$ or $0$ for each $i$.
    We have $\mathfrak{d}_{\chi}=\mathfrak{d}_{K/F}$ and $\mathfrak{d}_{\chi(\chi_{0}\circ N_{\mathbb{A}_{F}/\mathbb{A}_{\mathbb{Q}}})}$ divides $(139)\mathfrak{d}_{K/F}$. In particular, $\mathfrak{p}_{i}\mid\mathfrak{d}_{\chi(\chi_{0}\circ N_{\mathbb{A}_{F}/\mathbb{A}_{\mathbb{Q}}})}$ if and only if $\mathfrak{p}_{i}\mid\mathfrak{d}_{K/F}$.
     Without loss of generality,
    we may assume that $\mathfrak{p}_{1},\cdots,\mathfrak{p}_{t}\nmid\mathfrak{d}_{K/F}$ and
    $\mathfrak{p}_{t+1},\cdots\mathfrak{p}_{s}\mid\mathfrak{d}_{K/F}$. Notice that $\chi_{0}^{\ast}((37))=1$ and $\chi_{\rm f}(-1)=(-1)^{n}$. By Lemma \ref{lem:epsilon-factor} and Remark \ref{rem: Hilbert modular form}, we have
    \begin{eqnarray*}
        \epsilon_{\mathbf{f}\otimes\chi}&=&\epsilon_{\mathbf{f}^{\prime}\otimes\chi(\chi_{0}\circ N_{\mathbb{A}_{F}/\mathbb{A}_{\mathbb{Q}}})}\\
        &=&\epsilon_{\mathbf{f}^{\prime}}\cdot(\chi_{0}\circ N_{\mathbb{A}_{F}/\mathbb{A}_{\mathbb{Q}}})^{\ast}(\mathfrak{\mathfrak{p}}_{1}\cdots\mathfrak{\mathfrak{p}}_{t})\cdot(\chi_{0}\circ N_{\mathbb{A}_{F}/\mathbb{A}_{\mathbb{Q}}})_{\rm f}(-1)\cdot\chi^{\ast}(\mathfrak{p}_{1}\cdots\mathfrak{\mathfrak{p}}_{t})\chi_{\rm f}(-1)\cdot \prod_{i=t+1}^{s}(-\lambda(\mathfrak{p}_{i},\mathbf{f}^{\prime}))\\
        &=&\epsilon_{\mathbf{f}}\cdot\chi_{0}^{\ast}(N_{F/\mathbb{Q}}(\mathfrak{p}_{t+1}\cdots\mathfrak{p}_{s}))^{-1}\cdot(-1)^{t+n+(s-t)}\\
        &=&\epsilon_{\mathbf{f}}\cdot(-1)^{n+s}.
    \end{eqnarray*}
Clearly $\epsilon_{\mathbf{f}\otimes\chi}=\epsilon_{\mathbf{f}}$ if and only if $n+s$ is even.
\end{proof}

\section{Some estimations} \label{sec:esti}
\subsection{Number of lattice points}

\ \ \ \ \ Let $K$ be a quadratic imaginary extension of $F$. We assume that
$\mathfrak{o}^{\times}=\mathfrak{o}_{K}^{\times}$. Fix a set $\{\mathfrak{a}_{1},\cdots,\mathfrak{a}_{h_{F}}\}$ of fractional ideals that represents $\mathrm{Cl}_{F}$. We may assume that $\{\mathfrak{a}_{1},\cdots,\mathfrak{a}_{h^{\prime}}\}$ represents $\mathrm{Cl}_{F}/{\rm ker}(\phi)$ (see Section \ref{An estimate on $h$}).

Let $\sigma_{1},...,\sigma_{n}$ be the set of all embeddings from
$F$ into $\mathbb{R}$ ($n=[F:\mathbb{Q}]$). Then there is a natural
embedding
$$j:F\rightarrow\mathbb{R}^{n},\ a\mapsto (\sigma_{1}(a),...,\sigma_{n}(a)),$$
under which the image of any fractional ideal of $F$ is a complete lattice in
$\mathbb{R}^{n}$. For each $1\le i\le h_{F}$, we fix a fundamental domain $\mathbf{F}_{i}$ of $j(\mathfrak{a}_{i})$.

Similarly, we consider the embedding of the
multiplicative group $$\lambda: F^{\ast}\rightarrow\mathbb{R}^{n},\
a\mapsto ({\rm log}|\sigma_{1}(a)|,...,{\rm log}|\sigma_{n}(a)|).$$
Then $\lambda(\mathfrak{o}^{\times})$ is a complete lattice of the
hyperplane $$\mathbf{H}=\{(x_{1},...,x_{n})\in\mathbb{R}^{n}|\
x_{1}+...+x_{n}=0\}.$$ We fix a fundamental domain $\mathbf{U}$ of $\lambda(\mathfrak{o}^{\times})$, and let $d_{0}=d_{0}(\mathbf{U})$ be its diameter.

\begin{lem}\label{Lang}
    Let $\mathfrak{a}$ be a fractional ideal of $F$, $\mathbf{x}_{0}=(x_{0,1},...,x_{0,n})\in\mathbb{R}^{n}$ and $c_{1},...,c_{n}$ be positive real numbers.
For each $T_{0}>0$, put \begin{eqnarray*}
     C_{T_{0}}&:=&
        2n\cdot\max\limits_{i\in\{1,\cdots,h_{F}\}}\sup\limits_{\mathbf{E}\subseteq\mathbb{R}^{n},d(\mathbf{E})\le
        4\mathrm{e}^{\frac{\sqrt{n-1}d_{0}}{2\sqrt{n}}}(T_{0}|\mathfrak{a}_{i}|_{\mathbb{R}})^{1/n}\sqrt{n-1}}\#\{x\in\mathfrak{a}_{i}|\ (j(x)+\mathbf{F}_{i})\cap\mathbf{E}\ne\emptyset\},\end{eqnarray*}
where $d(\mathbf{E})$ denotes the diameter of $\mathbf{E}$.
\begin{enumerate} \item The constant $C_{T_{0}}$ depends only on $T_{0}$, $\mathbf{U}$ and the
choice of $(\mathfrak{a}_{i},\mathbf{F}_{i}),i=1,\cdots,h_{F}$.
\item If $c_{1}\cdots c_{n}\ge T_{0}|\mathfrak{a}|_{\mathbb{R}}$,
then
    \begin{equation} \label{eq:esti-latticepoint}
        \#\{x\in\mathfrak{a}|\ |\sigma_{j}(x)-x_{0,j}|\le c_{j}\ \forall
        j\}
        \le
            \bigg{(}\frac{2^{n}}{\sqrt{|d_{F}|}}+2nC_{T_{0}}T_{0}^{-1}\bigg{)}\frac{c_{1}...c_{n}}{|\mathfrak{a}|_{\mathbb{R}}}.
  \end{equation} \end{enumerate}
\end{lem}

\begin{proof}
The special case when $\mathbf{x}_{0}=0$ is treated in \cite[Theorem
3.7]{JN99}, except that the effective constant $C_{T_0}$ is not
clearly provided there. We will adjust its argument in loc. cit. to
show our effective bound in the general case.

    Let $a\in F^{\times}$ such that  $\mathfrak{a}(a)=\mathfrak{a}_{i}$ for some $1\le i\le h_{F}$.
    The vector $$\bigg{(}\mathrm{log}\frac{c_{1}\sigma_{1}(a)}{(c_{1}\sigma_{1}(a))^{1/n}\cdots (c_{n}\sigma_{n}(a))^{1/n}},
    \cdots,\mathrm{log}\frac{c_{n}\sigma_{n}(a)}{(c_{1}\sigma_{1}(a))^{1/n}\cdots (c_{n}\sigma_{n}(a))^{1/n}}\bigg{)}$$ lies in $\mathbf{H}$. Thus there exists $u\in\mathfrak{o}^{\times}$ such that $$\bigg{|}\mathrm{log}\frac{c_{j}\sigma_{j}(a)}{(c_{1}\sigma_{1}(a))^{1/n}\cdots (c_{n}\sigma_{n}(a))^{1/n}}-\mathrm{log}(\sigma_{j}u)\bigg{|}<\frac{\sqrt{n-1}d_{0}}{2\sqrt{n}}\ \forall 1\le j\le n.$$
Replacing $\mathfrak{a}$ by $\mathfrak{a}(au^{-1})=\mathfrak{a}(a)$,
$c_{j}$ by $c_{j}|\sigma_{j}(au^{-1})|$ and $x_{0,j}$ by
$x_{0,j}\sigma_{j}(au^{-1})\ (j=1,...,n)$ does not change the values
of both sides of (\ref{eq:esti-latticepoint}). So we may assume that
$\mathfrak{a}=\mathfrak{a}_{i}$ and
$$c_{j}\le\mathrm{e}^{\frac{\sqrt{n-1}d_{0}}{2\sqrt{n}}}(c_{1}\cdots
c_{n})^{1/n}\ (j=1,...,n).$$

Put $\mathbf{P}=\{\mathbf{x}\in\mathbb{R}^{n}|\ |x_{j}-x_{0,j}|\le
c_{j}\ \forall j\}$, and use $\partial(\mathbf{P})$ to denote its boundary. Then we have
    \begin{eqnarray*}
   &&\#\{x\in\mathfrak{a}_{i}|\ |\tau_{j}(x)-x_{0,j}|\le c_{j}\ \forall j\}\\
   &\le&\# \{x\in\mathfrak{a}_{i}|\ (j(x)+\mathbf{F}_{i})\cap\mathbf{P}\ne\emptyset\}\\
   &\le&\# \{x\in\mathfrak{a}_{i}|\ (j(x)+\mathbf{F}_{i})\subseteq\mathrm {int}(\mathbf{P})\}+\# \{x\in\mathfrak{a}_{i}|\ (j(x)+\mathbf{F}_{i})\cap\partial(\mathbf{P})\ne\emptyset\}\\
   &\le& \frac{\mathrm{vol}(\mathbf{P})}{\mathrm{vol}(\mathbf{F}_{i})}+\# \{x\in\mathfrak{a}_{i}|\ (j(x)+\mathbf{F}_{i})\cap\partial(\mathbf{P})\ne\emptyset\}   \end{eqnarray*}
\begin{equation}\label{eq:lattice points 1}
    \hskip -80pt =\frac{2^{n}c_{1}\cdots c_{n}}{\sqrt{|d_{F}|}|\mathfrak{a}_{i}|_{\mathbb{R}}}+\# \{x\in\mathfrak{a}_{i}|\ (j(x)+\mathbf{F}_{i})\cap\partial(\mathbf{P})\ne\emptyset\}.
\end{equation}

To estimate $\# \{x\in\mathfrak{a}_{i}|\
(j(x)+\mathbf{F}_{i})\cap\partial(\mathbf{P})\ne\emptyset\}$, we
consider the parametrization of $\mathbf{P}$:
    $$\varphi:\ [0,1]^{n}\rightarrow \mathbf{P},\ (t_{1},\cdots,t_{n})\mapsto\mathbf{x}_{0}+\Big{(}2c_{1}\Big{(}t_{1}-\frac{1}{2}\Big{)},\cdots,2c_{n}\Big{(}t_{n}-\frac{1}{2}\Big{)}\Big{)}. $$
    Let $\lVert \cdot \rVert$ be the euclidean norm on $\mathbb{R}^{n}$. Then we have
    \begin{equation}\label{eq:parametrization}
      \lVert\varphi(\mathbf{x})-\varphi(\mathbf{y})\rVert\le 2\mathrm{e}^{\frac{\sqrt{n-1}d_{0}}{2\sqrt{n}}}(c_{1}\cdots c_{n})^{1/n}\lVert \mathbf{x}-\mathbf{y}\rVert\ (\forall\ \mathbf{x},\mathbf{y}\in I^{n}).
    \end{equation}
Each  face of $\mathbf{P}$ is the image by $\varphi$ of a face of
$[0,1]^{n}$.

   If $c_{1}\cdots c_{n}\ge T_{0}|\mathfrak{a}_{i}|_{\mathbb{R}}$, we put $m=\big{[}\big{(}\frac{c_{1}...c_{n}}{T_{0}|\mathfrak{a}_{i}|_{\mathbb{R}}}\big{)}^{1/n}\big{]}\ge 1$.
One divides $[0,1]$ into the union of $m$ subintervals:
$$[0,1]=[0,\frac{1}{m} ]\cup [\frac{1}{m}, \frac{2}{m}]\cup \cdots
\cup [\frac{m-1}{m}, 1].$$ In this way, one can divide each face
$[0,1]^{n-1}$ of $[0,1]^{n}$ into $m^{n-1}$ smaller cubes. Each
smaller cube has diameter $\sqrt{n-1}/m$. Hence, by
(\ref{eq:parametrization}) the diameter of its image is upperly
bounded by
$$4\mathrm{e}^{\frac{\sqrt{n-1}d_{0}}{2\sqrt{n}}}(T_{0}|\mathfrak{a}_{i}|_{\mathbb{R}})^{1/n}\sqrt{n-1}.$$
Put
     $$C_{T_{0},i}:=\sup\limits_{\mathbf{E}\subseteq\mathbb{R}^{n},d(\mathbf{E})\le4\mathrm{e}^{\frac{\sqrt{n-1}d_{0}}{2\sqrt{n}}}(T_{0}|\mathfrak{a}_{i}|_{\mathbb{R}})^{1/n}\sqrt{n-1} }\#\{x\in\mathfrak{a}_{i}|\ (j(x)+\mathbf{F}_{i})\cap\mathbf{E}\ne\emptyset\}.$$
     Then for each face $[0,1]^{n-1}$, the number of $x\in\mathfrak{a}_{i}$ such that $(j(x)+\mathbf{F}_{i})\cap\varphi(I^{n-1})\ne\emptyset$ is bounded by $$C_{T_{0},i}m^{n-1}\le C_{T_{0},i}\Big{(}\frac{c_{1}...c_{n}}{T_{0}|\mathfrak{a}_{i}|_{\mathbb{R}}}\Big{)}^{1-1/n}\le C_{T_{0},i}\frac{c_{1}...c_{n}}{T_{0}|\mathfrak{a}_{i}|_{\mathbb{R}}}.$$  Then we have
     \begin{equation}\label{eq:lattice points 2}
       \# \{x\in\mathfrak{a}_{i}|\ (j(x)+\mathbf{F}_{i})\cap\partial(\mathbf{P})\ne\emptyset\}\le 2nC_{T_{0},i}\frac{c_{1}...c_{n}}{T_{0}|\mathfrak{a}_{i}|_{\mathbb{R}}}.    \end{equation}

Finally  (\ref{eq:lattice points 1}) and (\ref{eq:lattice points 2})
yield the conclusion.
\end{proof}

\begin{lem} \label{lem:norm estimation}Let $t>0$. We have the following.
    \begin{enumerate}
    \item \label{it-norm-a}Let $\mathfrak{N}$ be a fractional ideal of $K$ and $\mathcal{L}$ be a saturated invertible $\mathfrak{o}$-submodule of $\mathfrak{N}$ such that $|N_{K/\mathbb{Q}}(\mathcal{L}\mathfrak{o}_{K})|$ is minimal. We have
    $$\#\{[\alpha]\in(\mathfrak{N}\setminus\mathcal{L})/\mathfrak{o}^{\times}|
    \ |N_{K/\mathbb{Q}}(\alpha)|\le t |N_{K/\mathbb{Q}}(\mathfrak{N})| \}\le \frac{A_{1}t}{\sqrt{|\mathfrak{d}_{K/F}|_{\mathbb{R}}}}.$$ Here if we let $T_{0}=\frac{\pi^{n/2}{\rm e}^{\frac{\sqrt{n(n-1)}d_{0}}{2}}}{2^{n}\sqrt{|d_{F}|}}$, then
    $$A_{1}=2^{n-1}{\rm e}^{\sqrt{n(n-1)}d_{0}}\bigg{(}\frac{2^{n}}{\sqrt{|d_{F}|}}+2nC_{T_{0}}T_{0}^{-1}\bigg{)}\bigg{(}\frac{2^{n}}{\sqrt{|d_{F}|}}+2nC_{1}\bigg{)}.$$
    \item \label{it-norm-b}Let $\mathfrak{a}$ be a fractional ideal of $F$. We have
    $$\#\{[a]\in(\mathfrak{a}\setminus\{0\})/\mathfrak{o}^{\times}|\  |N_{F/\mathbb{Q}}(a)|\le t|\mathfrak{a}|_{\mathbb{R}}\}\le A_{2}t $$ where $$A_{2}=\mathrm{e}^{\frac{\sqrt{n(n-1)}d_{0}}{2}}\bigg{(}\frac{2^{n}}{\sqrt{|d_{F}|}}+2nC_{1}\bigg{)}.$$
    \end{enumerate}
\end{lem}

\begin{proof}
  We first prove  (\ref{it-norm-a}). Denote $t|N_{K/\mathbb{Q}}(\mathfrak{N})|$ by $t^{\prime}$. By Lemma \ref{thm:only-one} and Lemma \ref{lem:relation-d} we may assume that $$t^{\prime}\ge \frac{|\mathfrak{d}_{F/\mathbb{Q}}|_{\mathbb{R}}^{1/2}}{2^{n}}|N_{K/\mathbb{Q}}(\mathfrak{N})|.$$

    Translating $2\mathbf{U}$ (it is a fundamental domain for $\lambda((\mathfrak{o}^{\times})^{2})$) by some suitable vector $\mathbf{v}_{0}\in\mathbf{H}$, we obtain a subset $\mathcal{F}^{\prime}$ of $\mathbf{H}$ which is symmetric with respect to the origin. For any vector $\mathbf{v}=(v_{1},...,v_{n})\in\mathbb{R}^{n}$, we write ${\rm exp}(\mathbf{v})=({\rm e}^{v_{1}},...,{\rm e}^{v_{n}})$ and put $$y=(1/n,...,1/n),\ \mathcal{F}_{t^{\prime}}^{\prime}:=\{\mathbf{x}+s\mathbf{y}|\ \mathbf{x}\in\mathcal{F}^{\prime}, s\in(-\infty,{\rm log}t^{\prime}]\}\ \text{and}\ \mathcal{F}_{t^{\prime}}:={\rm exp}\mathcal{F}_{t^{\prime}}^{\prime}.$$ Then $\mathcal{F}_{t^{\prime}}$ lies in the first quadrant of $\mathbb{R}^{n}$. For any totally positive element $a\in F^{\ast}$, $|N_{F/\mathbb{Q}}(a)|\le t^{\prime}$ if and only if there exists (a unique) $v=u^{2}\in(\mathfrak{o}^{\times})^{2}$ such that $j(av)\in\mathcal{F}_{t^{\prime}}$. It follows that
    $$\#\{[\alpha]\in(\mathfrak{N}\setminus\mathcal{L})/\mathfrak{o}^{\times}|\ |N_{K/\mathbb{Q}}(\alpha)|\le t^{\prime}\}=\frac{1}{2}\#\{\alpha\in\mathfrak{N}\setminus\mathcal{L}|\ j(N_{K/F}(\alpha))\in\mathcal{F}_{t^{\prime}}\}.$$
    For any $\mathbf{a}=(a_{1},...,a_{n})\in\mathbb{R}^{n}_{+}$, we put $C_{\mathbf{a}}:=\{(x_{1},...,x_{n})\in\mathbb{R}^{n}|\ |x_{j}|\le a_{j}\ \forall j\}$. By a fundamental culculation, we see that $\mathcal{F}_{t^{\prime}}\subseteq C_{A\mathbf{t^{\prime}}^{1/n}}$ where $A={\rm e}^{\frac{\sqrt{n-1}d_{0}}{\sqrt{n}}}$ and $\mathbf{t}^{\prime1/n}=(t^{\prime1/n},...,t^{\prime1/n})$. We have
    $$\#\{[\alpha]\in(\mathfrak{N}\setminus\mathcal{L})/\mathfrak{o}^{\times}|
    \ |N_{K/\mathbb{Q}}(\alpha)|\le t^{\prime}\}\le\frac{1}{2}\#\{\alpha\in\mathfrak{N}\setminus\mathcal{L}|\ j(N_{K/F}(\alpha))\in C_{A\mathbf{t}^{\prime1/n}}\} .$$

    Since $\mathcal{L}$ is saturated, $\mathfrak{N}/\mathcal{L}$ is a projective $\mathfrak{o}$-module.
    Then we can write $\mathfrak{N}=\mathcal{L}\oplus\mathcal{L}^{\prime}$ for $\mathcal{L}=\mathfrak{a}\alpha_{1}$,
    $\mathcal{L}^{\prime}=\mathfrak{b}\alpha_{2}$. Let
    $\alpha=x\alpha_{1}+y\alpha_{2}$ be in $\mathfrak{N}\setminus\mathcal{L}$ with $j(N_{K/F}(\alpha))\in
    C_{A\mathbf{t}^{\prime1/n}}$. We put $N_{K/F}(x\alpha_{1}+y\alpha_{2})=ax^{2}+bxy+cy^{2}$ and
$\varDelta=4ac-b^{2}$, and fix notations
    \begin{eqnarray*}
        \mathbf{y}_{0}&=j\Big{(}\frac{by}{2a}\Big{)},\ \mathbf{z}_{0}=\frac{1}{2j(a)}\sqrt{4Aj(a)\mathbf{t}^{\prime1/n}-j(\varDelta y^{2})},\ \mathbf{z}_{1}=\sqrt{\frac{4Aj(a)\mathbf{t}^{\prime1/n}}{j(\varDelta)}},\\
        &\text{and}\
        \mathbf{z}_{2}=\sqrt{\frac{A\mathbf{t}^{\prime1/n}}{j(a)}}.
    \end{eqnarray*}
   Here (and in the rest of the proof) all the operations on vectors are defined in terms of
   components. From the relation
$$j(N_{K/F}( x\alpha_{1}+y\alpha_{2})) = \frac{4j(a)^2 (j(x)+\mathbf{y}_0)^2 +  j(\varDelta y^2)}{4j(a)}, $$
  we obtain that $$j(N_{K/F}(x\alpha_{1}+y\alpha_{2}))\in C_{A\mathbf{t}^{\prime1/n}}\ \text{if and only if}\ j(x)+\mathbf{y}_{0}\in
  C_{\mathbf{z}_{0}}.$$ Moreover, we have $j(y)\in C_{\mathbf{z}_1}$ and so $|N_{F/\mathbb{Q}}(y)|\leq |N_{\mathbb{R}^{n}}(\mathbf{z}_1)|$.
  Since $\alpha\notin\mathcal{L}$, we have $y\neq 0$. Thus $|N_{F/\mathbb{Q}}(y)|\ge |\mathfrak{b}|_{\mathbb{R}}$.

    Let $d_{K}$ (resp. $d_{F}$) denote the absolute discriminant of $K$ (resp. $F$). Then $|d_{K}|=|d_{F}|^{2}|\mathfrak{d}_{K/F}|_{\mathbb{R}}$. By Minkowski's theory
    \cite[Ch.1, Lemma 6.2]{JN99}, the minimality of $|N_{K/\mathbb{Q}}(\mathcal{L}\mathfrak{o}_{K})|$ implies that
    $$ |N_{F/\mathbb{Q}}(a)|\le\Big{(}\frac{2}{\pi}\Big{)}^{n}\sqrt{|d_{K}|}|\mathfrak{a}|_{\mathbb{R}}^{-2}|N_{K/\mathbb{Q}}(\mathfrak{N})|.$$ Then one can deduce that $$\frac{N_{\mathbb{R}^{n}}(\mathbf{z}_{2})}{|\mathfrak{a}|_{\mathbb{R}}}\ge\frac{\pi^{n/2}A^{n/2}}{2^{n}\sqrt{|d_{F}|}}.$$
    Applying Lemma \ref{Lang} with $T_{0}:=\frac{\pi^{n/2}A^{n/2}}{2^{n}\sqrt{|d_{F}|}}$, we have
    \begin{eqnarray*}
        \#\{x\in\mathfrak{a}|\ j(x)+\mathbf{y}_{0}\in C_{\mathbf{z}_{0}}\}&\le&\#\{x\in\mathfrak{a}|\ j(x)+\mathbf{y}_{0}\in C_{\mathbf{z}_{2}}\} \\
        &\le&   \bigg{(}\frac{2^{n}}{\sqrt{|d_{F}|}}+2nC_{T_{0}}T_{0}^{-1}\bigg{)}\frac{N_{\mathbb{R}^{n}}(\mathbf{z}_{2})}{|\mathfrak{a}|_{\mathbb{R}}}.
    \end{eqnarray*}
Again by Lemma \ref{Lang} with $T_0=1$,
   \begin{eqnarray*}
        \#\{y\in\mathfrak{b}|\ j(y)\in C_{\mathbf{z}_{1}}\setminus\{0\}\}&\le&  \bigg{(}\frac{2^{n}}{\sqrt{|d_{F}|}}+2nC_{1}\bigg{)}\frac{N_{\mathbb{R}^{n}}(\mathbf{z}_{1})}{|\mathfrak{b}|_{\mathbb{R}}}\\
    \end{eqnarray*}
    It is easy to see that
    $$\frac{N_{\mathbb{R}^{n}}(\mathbf{z}_{2})}{|\mathfrak{a}|_{\mathbb{R}}}\cdot\frac{N_{\mathbb{R}^{n}}(\mathbf{z}_{1})}
    {|\mathfrak{b}|_{\mathbb{R}}}=\frac{2^{n}A^{n}t^{\prime}}{\sqrt{|\varDelta|_{\mathbb{R}}}|\mathfrak{ab}|_{\mathbb{R}}}
    =\frac{2^{n}A^{n}t^{\prime}}{\sqrt{|\mathfrak{d}_{K/F}|_{\mathbb{R}}}|N_{K/\mathbb{Q}}(\mathfrak{N})|}=\frac{2^{n}A^{n}t}{\sqrt{|\mathfrak{d}_{K/F}|_{\mathbb{R}}}}.$$
    Therefore, we have
    \begin{eqnarray*}
        & &\#\{[\alpha]\in(\mathfrak{N}\setminus\mathcal{L})/\mathfrak{o}^{\times}|\ |N_{K/\mathbb{Q}}(\alpha)|\le t|N_{K/\mathbb{Q}}(\mathfrak{N})|\}
        \\ &\leq & \frac{1}{2}\#\{x\in\mathfrak{a}|\ j(x)+\mathbf{y}_{0}\in C_{\mathbf{z}_{0}}\} \cdot \#\{y\in\mathfrak{b}|\ j(y)\in C_{\mathbf{z}_{1}}\setminus\{0\}\}
        \\ &\leq &
        \frac{1}{2}\bigg{(}\frac{2^{n}}{\sqrt{|d_{F}|}}+2nC_{T_{0}}T_{0}^{-1}\bigg{)}\frac{N_{\mathbb{R}^{n}}(\mathbf{z}_{2})}{|\mathfrak{a}|_{\mathbb{R}}}
        \cdot \bigg{(}\frac{2^{n}}{\sqrt{|d_{F}|}}+2nC_{1}\bigg{)}\frac{N_{\mathbb{R}^{n}}(\mathbf{z}_{1})}{|\mathfrak{b}|_{\mathbb{R}}}
        \\
        &\le&\frac{1}{2}\cdot\bigg{(}\frac{2^{n}}{\sqrt{|d_{F}|}}+2nC_{T_{0}}T_{0}^{-1}\bigg{)}\bigg{(}\frac{2^{n}}{\sqrt{|d_{F}|}}+2nC_{1}\bigg{)}\times\frac{2^{n}A^{n}t}{\sqrt{|\mathfrak{d}_{K/F}|_{\mathbb{R}}}}\\
        &=&\frac{A_{1}t}{\sqrt{|\mathfrak{d}_{K/F}|_{\mathbb{R}}}}.
    \end{eqnarray*}

Next we prove (\ref{it-norm-b}). Write $|\mathfrak{a}|_{\mathbb{R}}t$ by $t^{\prime}$ again. Replacing
$\lambda((\mathfrak{o}^{\times})^{2})$ by
$\lambda(\mathfrak{o}^{\times})$, an argument similar to the proof
of (\ref{it-norm-a}) implies that
$$\#\{[a]\in(\mathfrak{a}\setminus\{0\})/\mathfrak{o}^{\times}|\  |N_{F/\mathbb{Q}}(a)|\le t|\mathfrak{a}|_{\mathbb{R}}\}\le\#\{a\in\mathfrak{a}\setminus\{0\}|\ j(a)\in C_{A^{1/2}\mathbf{t}^{\prime1/n}}\}. $$
Since $a\ne0$, we have $|N_{F/\mathbb{Q}}(a)|\ge
|\mathfrak{a}|_{\mathbb{R}}$. Then we may assume that $(A^{n/2}t^{\prime}\ge) t^{\prime}\ge |\mathfrak{a}|_{\mathbb{R}}$.
Assertion (\ref{it-norm-b}) follows from Lemma \ref{Lang} directly.
\end{proof}


\subsection{Measures and their Mellin transforms}

We recall some results of \cite{JO 84}.

Let  $\mu_{1},...,\mu_{r}$ be a family of positive Borel regular
measures on $\mathbb{R}^{\ast}_{+}$ with respect to which
$t^{-\sigma}$ ($\sigma>0$) is measurable. Let $x$ be a positive
number and $s$ be a complex number with ${\rm Re}(s)=\sigma$.  For
each $k\in \{1\cdots,  r\}$, let $\hat{\mu}_k$ be the Mellin
transform of $\mu_k$ defined by
$$
\hat{\mu}_{k}(s)=\int_{\mathbb{R}^{\ast}_{+}}t^{-s}\mu_{k}.$$ Put
$$
J(\mu_{1},...,\mu_{r};x)=\int_{\sigma-i\infty}^{\sigma+i\infty}\hat{\mu}_{1}(s)...\hat{\mu}_{r}(s)x^{-s}\bigg{(}s-\frac{1}{2}\bigg{)}^{-3}\frac{{\rm
d}s}{2\pi i}. $$

\begin{prop}\label{prop:measures}(\cite[3.3, Lemma 2 and Lemma 3]{JO 84})
 \begin{enumerate} \item  $J(\mu_{1},...,\mu_{r};x)\ (x>0)$ is a positive
 function.
\item  Let $\mu^{\prime}_{1},...,\mu^{\prime}_{r}$ be another family of measures with the same assumptions.  If
    $$\int_{0}^{x}\mu_{k}([0,t]){\rm dt}\le\int_{0}^{x}\mu^{\prime}_{k}([0,t]){\rm
    dt}$$ for each $k=1,\cdots, r$,
    then $$J(\mu_{1},...,\mu_{r};x)\le
    J(\mu^{\prime}_{1},...,\mu^{\prime}_{r};x).$$
    \end{enumerate}\end{prop}

 In the rest of the paper, for each $1\le i\le h$, we fix a saturated invertible $\mathfrak{o}$-submodule $\mathcal{L}_{i}$
of $\mathfrak{N}_{i}$ such that $|N_{K/\mathbb{Q}}(\mathcal{L}_{i}\mathfrak{o}_{K})|$ is minimal.

For each $x>0$, let $\delta(x)$ denote the Dirac measure
concentrated at $x$.

\begin{lem}\label{lem:measures}
    \begin{enumerate}
    \item\label{it-measures-1} If $\mu={\rm e}^{-t^{-1}}t^{u-1}{\rm d}t$, then $\hat{\mu}(s)=\varGamma(s-u)$.
    \item\label{it-measures-2}
     We write $t_{0}=\frac{\sqrt{|\mathfrak{d}_{K/F}
    |_{\mathbb{R}}}}{2^{n}}$. Put $$\mu_{K}=\sum_{i=1}^{h}\sum_{\mathcal{L}\in\mathscr{L}_{i},\mathcal{L}\cap\mathcal{L}_{i}=0}
    \delta(|N_{K/\mathbb{Q}}(\mathcal{L}\mathfrak{N}_{i}^{-1})|)$$ and $$ \mu^{\prime}_{K}=\frac{A_{1}h_{K}}{\sqrt{|\mathfrak{d}_{K/F}
    |_{\mathbb{R}}}}{\rm d}t\bigg{|}_{[t_{0},\infty]}+\frac{A_{1}h_{K}}{2^{n}}\delta(t_{0}).$$
    Then for every $x>0$ we have
    \begin{equation}\label{eq:measure-b}\int_{0}^{x}\mu_{K}([0,t]){\rm dt}\le\int_{0}^{x}\mu^{\prime}_{K}([0,t]){\rm
    dt}.\end{equation}
    Moreover, we have
    $$\hat{\mu}_{K}(s)=\sum_{i=1}^{h}\sum_{\mathcal{L}\in\mathscr{L}_{i},\mathcal{L}\cap\mathcal{L}_{i}=0}\bigg{|}N_{K/\mathbb{Q}}\bigg{(}\frac{\mathcal{L}\mathfrak{o}_{K}}{\mathfrak{N}_{i}}\bigg{)}\bigg{|}^{-s}$$ and $$\hat{\mu}^{\prime}_{K}(s)=h_{K}\frac{A_{1}2^{ns-n}s}{(s-1)\sqrt{|\mathfrak{d}_{K/F}|_{\mathbb{R}}}^{s}}.$$
\item\label{it-measures-3}
Put
$$\mu_{F}=\sum_{\mathfrak{m}\subseteq\mathfrak{o}}\delta(|N_{F/\mathbb{Q}}(\mathfrak{m})|^{2})$$
and $$\mu^{\prime}_{F}=\frac{h_{F}\sqrt{A_{2}}}{2\sqrt{t}}{\rm
dt}\big{|}_{[1,\infty]}+\sqrt{A_{2}}h_{F}\delta(1).$$ Then for every
$x>0$ we have
\begin{equation}\label{eq:measure-c} \int_{0}^{x}\mu_{F}([0,t]){\rm
dt}\le\int_{0}^{x}\mu^{\prime}_{F}([0,t]){\rm dt}.\end{equation}
Moreover,
$$\hat{\mu}_{F}(s)=\zeta_{F}(2s) \ \text{ and } \ \ \ \hat{\mu}^{\prime}_{F}(s)=\frac{h_{F}\sqrt{A_{2}}s}{s-\frac{1}{2}}.        $$
        \end{enumerate}
\end{lem}
\begin{proof}
   Assertion (\ref{it-measures-1}) is clear.

Lemma \ref{lem:ideal} implies that
    $$\mu_{K}=\sum_{i=1}^{h}\sum_{j=1}^{h^{\prime}}\sum_{[\alpha]\in(\mathfrak{a}_{j}^{-1}\mathfrak{N}_{i}
    \setminus\mathfrak{a}_{j}^{-1}\mathcal{L}_{i})/\mathfrak{o}^{\times}}\delta(|N_{K/\mathbb{Q}}(\alpha\mathfrak{a}_{j}\mathfrak{N}_{i}^{-1})|). $$
By cosidering (for all $i,j$)
$\mathfrak{N}=\mathfrak{a}_{j}^{-1}\mathfrak{N}_{i}$ and
$\mathcal{L}=\mathfrak{a}_{j}^{-1}\mathcal{L}_{i}$ in Lemma
\ref{lem:norm estimation} (\ref{it:idea-form-a}), we have for every
$t>0$,
$$\mu_{K}([0,t])\le h_{K}\frac{A_{1}t}{\sqrt{|\mathfrak{d}_{K/F}|_{\mathbb{R}}}} \le \mu'_{K}([0,t]) ,$$ which yields (\ref{eq:measure-b}). The equalities for $\hat{\mu}_K$ and $\hat{\mu}'_K$
follow  from  direct calculations.

Similarly,  (\ref{eq:measure-c}) follows from Lemma \ref{lem:norm
estimation} (\ref{it-norm-b}). The equalities for $\hat{\mu}_F$ and
$\hat{\mu}'_F$ follow from direct calculations.
\end{proof}

Let $\varPsi(s)$ and $\varPhi(s)$ satisfy Condition \ref{condition}. As in \cite{JO 84}, for
every real number $U>1$, we put
$$\varPhi(U,s)=\prod_{\mathfrak{p}:|\mathfrak{p}|_{\mathbb{R}}<U}\varPhi_{\mathfrak{p}}(s)\hskip10pt\text{and}\hskip10pt
\varPhi(U^{\ast},s)=\varPhi(s)-\varPhi(U,s).$$
We write $\gamma(s)=M^{s}\varGamma(s)^{2n}.$ When $\sigma>1$, we put
\begin{eqnarray*}
&&J(U)=\int_{\sigma-i\infty}^{\sigma+i\infty}|\mathfrak{d}_{K/F}|_{\mathbb{R}}^{s-1/2}\gamma\bigg{(}s+\frac{1}{2}\bigg{)}
\varPsi\bigg{(}s+\frac{1}{2}\bigg{)}\varPhi\bigg{(}U,s+\frac{1}{2}\bigg{)}\bigg{(}s-\frac{1}{2}\bigg{)}^{-3}\frac{{\rm d}s}{2\pi i}\\
&&  J(U^{\ast})=\int_{\sigma-i\infty}^{\sigma+i\infty}|\mathfrak{d}_{K/F}|_{\mathbb{R}}^{s-1/2}\gamma\bigg{(}s+\frac{1}{2}\bigg{)}\varPsi\bigg{(}s+\frac{1}{2}\bigg{)}\varPhi\bigg{(}U^{\ast},s+\frac{1}{2}\bigg{)}\bigg{(}s-\frac{1}{2}\bigg{)}^{-3}\frac{{\rm d}s}{2\pi i}.
\end{eqnarray*}
Let $\varDelta$ be the oriented path
$$\bigg{(}\frac{1}{2}-i\infty,\frac{1}{2}-i\eta^{\prime}\bigg{]}+\bigg{[}\frac{1}{2}-i\eta^{\prime},\frac{1}{2}-\eta-i\eta^{\prime}\bigg{]}+\bigg{[}\frac{1}{2}-\eta-i\eta^{\prime},\frac{1}{2}-\eta+i\eta^{\prime}\bigg{]}+\bigg{[}\frac{1}{2}-\eta+i\eta^{\prime},\frac{1}{2}+i\eta^{\prime}\bigg{]}+\bigg{[}\frac{1}{2}+i\eta^{\prime},\frac{1}{2}+i\infty\bigg{)}$$
from $\frac{1}{2}-i\infty$ to $\frac{1}{2}+i\infty$. Here $\eta$, $\eta^{\prime}\in(0,1/4)$ are sufficiently small such that $\varPsi(s)$ is holomorphic on the region between $\varDelta$ and $\{{\rm Re}(s)=\sigma\}$. Put
$$J_{1}(U):=\int_{\varDelta}|\mathfrak{d}_{K/F}|_{\mathbb{R}}^{s-1/2}\gamma\bigg{(}s+\frac{1}{2}\bigg{)}
\varPsi\bigg{(}s+\frac{1}{2}\bigg{)}\varPhi\bigg{(}U,s+\frac{1}{2}\bigg{)}\bigg{(}s-\frac{1}{2}\bigg{)}^{-3}\frac{{\rm d}s}{2\pi i}.$$

\begin{lem}\label{lem: J(U) and J(U*)}
    Suppose that $\varPhi(s)$ and $\varPsi(s)$ satisfy Condition \ref{condition}. Then we have the following.
    \begin{enumerate}
        \item\label{it-J(U)-1}$J(U)=-J(U^{\ast})$.
        \item\label{it-J(U)-2}We have
        $$J(U)=G_{1}(M,\varPsi)\varPhi(U,1)\bigg{(}{\rm log}|\mathfrak{d}_{K/F}|_{\mathbb{R}}+\frac{\varPhi^{\prime}(U,1)}{\varPhi(U,1)}+G_{2}(M,\varPsi)\bigg{)}+J_{1}(U)$$ where $G_{1}(M,\varPsi)=\frac{\gamma(s)\varPsi(s)}{s-1}\big{|}_{s=1}$ and $G_{2}(M,\varPsi)=\frac{\rm dlog}{\rm dt}\big{(}\frac{\gamma(s)\varPsi(s)}{s-1}\big{)}\big{|}_{s=1}$. Moreover if we put $$G_{3}(M,\varPsi)=\int_{\varDelta}\bigg{|}\gamma\bigg{(}s+\frac{1}{2}\bigg{)}\varPsi\bigg{(}s+\frac{1}{2}\bigg{)}\bigg{(}s-\frac{1}{2}\bigg{)}^{-3}\bigg{|}\frac{\rm |ds|}{2\pi},$$ then
        $|J_{1}(U)|\le G_{3}(M,\varPsi){\rm sup}_{\varDelta}|\varPhi(U,s+\frac{1}{2})|$. 
        \end{enumerate}

Moreover, if $\varPhi(s)$ and $\varPsi(s)$ are constructed in terms of ($\mathbf{f},\chi$) in Lemma \ref{lem:hilbert modular form}, then $$|G_{1}(M,\varPsi)|\ge G_{1},\hskip10pt|G_{2}(M,\varPsi)|\le G_{2}\hskip10pt\text{and}\hskip10pt G_{3}(M,\varPsi)\le G_{3}.$$ Here if we let $\zeta_{F,\mathfrak{a}}(s)=\prod_{\mathfrak{p}\nmid\mathfrak{a}}(1-|\mathfrak{p}|_{\mathbb{R}}^{-s})^{-1}$, then $G_{1}$, $G_{2}$ and $G_{3}$ are given as follows:
$$  G_{1}=\bigg{|}\frac{2|\mathfrak{d}^{2}|_{\mathbb{R}}}{(2\pi)^{2n}|\mathfrak{a}^{n}|_{\mathbb{R}}}\times L_{\rm f}(1,\vee^{2}\pi_{\mathbf{f}})\times(\zeta_{F,\mathfrak{a}}^{-1})^{\prime}(1)\times\prod_{\mathfrak{p}^{2}\mid\mathfrak{a}}\frac{|\mathfrak{p}|_{\mathbb{R}}}{(\sqrt{|\mathfrak{p}|_{\mathbb{R}}}+1)^{2}}\bigg{|},$$
$$ G_{2}=\bigg{|}{\rm log}\bigg{(} \frac{|\mathfrak{ad}^{2}|_{\mathbb{R}}}{(2\pi)^{2n}} \bigg{)}\bigg{|}+\frac{2n\varGamma^{\prime}(1)}{\varGamma(1)}+\bigg{|}\frac{2L_{\mathbf{f}}^{\prime}(1,\vee^{2}\pi_{\mathbf{f}})}{L_{\mathbf{f}}(1,\vee^{2}\pi_{\mathbf{f}})}\bigg{|}+\bigg{|}\frac{(\zeta_{F,\mathfrak{a}}^{-1})^{\prime\prime}(1)}{(\zeta_{F,\mathfrak{a}}^{-1})^{\prime}(1)}\bigg{|}+\sum_{\mathfrak{p}^{2}\mid\mathfrak{a}}\frac{2\sqrt{|\mathfrak{p}|_{\mathbb{R}}}+2}{(\sqrt{|\mathfrak{p}|_{\mathbb{R}}}-1)^{2}}{\rm log}|\mathfrak{p}|_{\mathbb{R}} $$    
and
$$ G_{3}={\rm max}\Bigg{\{}\frac{|\mathfrak{ad}^{2}|_{\mathbb{R}}}{(2\pi)^{2n}},\frac{|\mathfrak{d}|^{\frac{3}{2}}_{\mathbb{R}}}{(2\pi)^{\frac{3n}{2}}|\mathfrak{a}|^{\frac{3n}{4}}_{\mathbb{R}}}\Bigg{\}}\cdot\prod_{\mathfrak{p}^{2}\mid\mathfrak{a}}\frac{\sqrt{|\mathfrak{p}|_{\mathbb{R}}}}{(|\mathfrak{p}|^{\frac{1}{4}}_{\mathbb{R}}-1)^{2}}\times\int_{\varDelta}\bigg{|}\varGamma\Big{(}s+\frac{1}{2}\Big{)}^{2n} \frac{L_{\mathbf{f}}(2s,\vee^{2}\pi_{\mathbf{f}})}{\zeta_{F,\mathfrak{a}}(2s)}\cdot\bigg{(}s-\frac{1}{2}\bigg{)}^{-3}\bigg{|}\frac{\rm |ds|}{2\pi}.$$
$G_{1}$, $G_{2}$ and $G_{3}$ depend only on $\mathbf{f}$.

\end{lem}
\begin{proof}
    The proof is similar to \cite[3.2]{JO 84} by using Condition \ref{condition}. We focus on the estimations of $G_{1}(M,\varPsi)$, $G_{2}(M,\varPsi)$ and $G_{3}(M,\varPsi)$. If $\varPhi$ and $\varPsi$ are constructed by ($\mathbf{f},\chi$) in Lemma \ref{lem:hilbert modular form}, we have
    $$\varPsi(s)=L_{\mathbf{f}}(2s-1,\vee^{2}\pi_{\mathbf{f}})\times\zeta_{F,\mathfrak{a}}(2s-1)^{-1}\times\prod_{\mathfrak{p}^{2}\mid{\rm g.c.d}(\mathfrak{a},\mathfrak{d}_{\chi}^{2})}D(s,\mathbf{f}\otimes\chi)_{\mathfrak{p}}.$$
     Then $$G_{1}(M,\varPsi)=2\gamma(1)\times L_{\rm f}(1,\vee^{2}\pi_{\mathbf{f}})\times(\zeta_{F,\mathfrak{a}}^{-1})^{\prime}(1)\times\prod_{\mathfrak{p}^{2}\mid{\rm g.c.d}(\mathfrak{a},\mathfrak{d}_{\chi}^{2})}D(1,\mathbf{f}\otimes\chi)_{\mathfrak{p}}$$ and 
    $$G_{2}(M,\varPsi)=\frac{\gamma^{\prime}(1)}{\gamma(1)}+\frac{2L_{\mathbf{f}}^{\prime}(1,\vee^{2}\pi_{\mathbf{f}})}{L_{\mathbf{f}}(1,\vee^{2}\pi_{\mathbf{f}})}+\frac{(\zeta_{F,\mathfrak{a}}^{-1})^{\prime\prime}(1)}{(\zeta_{F,\mathfrak{a}}^{-1})^{\prime}(1)}+\sum_{\mathfrak{p}^{2}\mid{\rm g.c.d}(\mathfrak{a},\mathfrak{d}_{\chi}^{2})}\frac{D^{\prime}(1,\mathbf{f}\otimes\chi)_{\mathfrak{p}}}{D(1,\mathbf{f}\otimes\chi)_{\mathfrak{p}}}.$$ 
    
    We decompose $\mathfrak{d}_{K/F}$ into a product $\mathfrak{d}_{K/F,1}\cdot\mathfrak{d}_{K/F,2}$. Here for each $\mathfrak{p}\mid\mathfrak{d}_{K/F,1}$, $v_{\mathfrak{p}}(\mathfrak{a})\le1$, while for each $\mathfrak{p}\mid\mathfrak{d}_{K/F,2}$, $\mathfrak{p}^{2}\mid\mathfrak{a}$. Then we have
    $$|\mathfrak{d}_{K/F,1}^{2}|_{\mathbb{R}}\le |\mathfrak{a}_{\chi}|_{\mathbb{R}}\le |\mathfrak{ad}_{K/F}^{2}|_{\mathbb{R}}.$$
     By Corollary \ref{cor: prime power in discriminant }, each prime divisor of $\mathfrak{d}_{K/F,2}$ has multiplicity $\le 2e_{\mathfrak{p}}+1(\le 2n+1)$. Then 
     $|\mathfrak{d}_{K/F,2}|_{\mathbb{R}}\le |\mathfrak{a}|_{\mathbb{R}}^{n+\frac{1}{2}}$. Therefore we obtain
    \begin{equation}\label{eq:estimation of M} \frac{|\mathfrak{d}^{2}|_{\mathbb{R}}}{(2\pi)^{2n}|\mathfrak{a}^{n}|_{\mathbb{R}}}\le\frac{\sqrt{|\mathfrak{ad}^{4}|_{\mathbb{R}}}}{(2\pi)^{2n}|\mathfrak{d}_{K/F,2}|_{\mathbb{R}}}\le M=\frac{\sqrt{|\mathfrak{a}\mathfrak{a}_{\chi}\mathfrak{d}^{4}|_{\mathbb{R}}}}{(2\pi)^{2n}|\mathfrak{d}_{K/F}|_{\mathbb{R}}}\le \frac{|\mathfrak{ad}^{2}|_{\mathbb{R}}}{(2\pi)^{2n}}.
    	\end{equation}
    Moreover, by elementary calculations (by using the facts in Section \ref{subsection 2.1}) we have
    $$  |D(1,\mathbf{f}\otimes\chi)_{\mathfrak{p}}|\ge\frac{|\mathfrak{p}|_{\mathbb{R}}}{(\sqrt{|\mathfrak{p}|_{\mathbb{R}}}+1)^{2}}$$ 
    $${\rm sup}_{\varDelta}\Big{|}D\Big{(}s+\frac{1}{2},\mathbf{f}\otimes\chi\Big{)}_{\mathfrak{p}}\Big{|}\le \frac{\sqrt{|\mathfrak{p}|_{\mathbb{R}}}}{(|\mathfrak{p}|^{\frac{1}{4}}_{\mathbb{R}}-1)^{2}}\hskip10pt \text{(we have used the assumption  $\eta<\frac{1}{4}$)} $$
    and
    $$   \bigg{|}\frac{D^{\prime}(1,\mathbf{f}\otimes\chi)_{\mathfrak{p}}}{D(1,\mathbf{f}\otimes\chi)_{\mathfrak{p}}}\bigg{|}\le\frac{2\sqrt{|\mathfrak{p}|_{\mathbb{R}}}+2}{(\sqrt{|\mathfrak{p}|_{\mathbb{R}}}-1)^{2}}{\rm log}|\mathfrak{p}|_{\mathbb{R}}$$
  for all $\mathfrak{p}\mid {\rm g.c.d}(\mathfrak{a},\mathfrak{d}_{K/F}^{2})$. Combining all the above together, we deduce the last assertion. 
\end{proof}


\begin{prop}\label{prop: estimation of J(U*)}
 Suppose $\varPsi(s)$ and $\varPhi(s)$ are constructed by $(\mathbf{f},\chi)$ in Lemma \ref{lem:hilbert modular form}. Put $M^{\prime}=(2\pi)^{-2n}|\mathfrak{ad}^{2}|_{\mathbb{R}}$. Then we have
    $$|J(U^{\ast})|\le\frac{B_{1}h^{2}_{K}}{\sqrt{U}}\bigg{(}\bigg{|}{\rm log}\frac{M^{\prime}|\mathfrak{d}_{K/F}|_{\mathbb{R}}}{U}\bigg{|}+2n+6\bigg{)}^{4}+B_{2}h_{K}\sum_{i=1}^{h}\bigg{|}N_{K/\mathbb{Q}}\Big{(}\frac{\mathcal{L}_{i}\mathfrak{o}_{K}}{\mathfrak{N}_{i}}\Big{)}\bigg{|}^{-1}+B_{3}\frac{h_{K}^{2}}{\sqrt{|\mathfrak{d}_{K/F}|_{\mathbb{R}}}}$$
where $$B_{1}=\frac{A_{2}h_{F}^{2}M^{\prime}}{96},\hskip10pt B_{2}=16A_{1}M^{\prime3/2}\varGamma\Big{(}\frac{3}{2}\Big{)}\zeta_{F}(2)\hskip10pt\text{and}\hskip10pt B_{3}=4A_{1}^{2}{\rm e}M^{\prime3/2}{\rm max}\big{\{}2, {\rm
log}(4^{n}M^{\prime})\big{\}}.$$
\end{prop}

\begin{proof}
    The argument is similar to the proof of \cite[Proposition 1]{JO 84}. By Condition \ref{condition} (\ref{condition-1}) and assertions (\ref{it:estimation-1}), (\ref{it:estimation-2}) of Proposition \ref{prop:estimation of h} , we have
    $$\varPsi\Big{(}s+\frac{1}{2}\Big{)}<<\zeta_{F}(2s)^{2}$$ and
    $$\varPhi\bigg{(}U^{\ast},s+\frac{1}{2}\bigg{)}<<\bigg{(}\sum_{i=1}^{h}\sum_{\mathcal{L}\in\mathscr{L}_{i}^{sa}}\bigg{|}N_{K/\mathbb{Q}}\bigg{(}
    \frac{\mathcal{L}\mathfrak{o}_{K}}{\mathfrak{N}_{i}}\bigg{)}\bigg{|}^{-s}\bigg{)}^{2}-\sum_{\big{|}N_{K/\mathbb{Q}}\big{(}\frac{\mathcal{L}_{i}
    \mathcal{L}_{j}\mathfrak{o}_{K}}{\mathfrak{N}_{i}\mathfrak{N}_{j}}\big{)}\big{|}<U}\bigg{|}N_{K/\mathbb{Q}}\bigg{(}\frac{\mathcal{L}_{i}\mathcal{L}_{j}
    \mathfrak{o}_{K}}{\mathfrak{N}_{i}\mathfrak{N}_{j}}\bigg{)}\bigg{|}^{-s}.$$
    Then by Lemma \ref{lem:measures} (\ref{it-measures-2}) and
    (\ref{it-measures-3}) we get
$$
\hskip -50pt \varPsi\bigg{(}s+\frac{1}{2}\bigg{)}\varPhi\bigg{(}U^{\ast},s+\frac{1}{2}\bigg{)}<<\Bigg{(}\hat{\mu}_{F}(s)\sum_{i=1}^{h}\bigg{|}N_{K/\mathbb{Q}}\bigg{(}\frac{\mathcal{L}_{i}\mathfrak{o}_{K}}{\mathfrak{N}_{i}}\bigg{)}\bigg{|}^{-s}+\hat{\mu}_{K}(s)\Bigg{)}^{2}\\
$$
\begin{equation}\label{eq:esti-diri} \hskip 140pt -\hat{\mu}_{F}(s)^{2}\times\sum_{\big{|}N_{K/\mathbb{Q}}\big{(}\frac{\mathcal{L}_{i}\mathcal{L}_{j}\mathfrak{o}_{K}}{\mathfrak{N}_{i}\mathfrak{N}_{j}}\big{)}\big{|}<U}\bigg{|}N_{K/\mathbb{Q}}\bigg{(}\frac{\mathcal{L}_{i}\mathcal{L}_{j}\mathfrak{o}_{K}}{\mathfrak{N}_{i}\mathfrak{N}_{j}}\bigg{)}\bigg{|}^{-s}.
        \end{equation}
Denote the measure ${\rm e}^{-t^{-1}}t^{-3/2}{\rm d}t$ by $\mu$.
Then $\hat{\mu}(s)=\varGamma(s+\frac{1}{2})$. It's easy to see that
$$\int_{\sigma-i\infty}^{\sigma+i\infty}|\mathfrak{d}_{K/F}|_{\mathbb{R}}^{s-1/2}\gamma\bigg{(}s+\frac{1}{2}\bigg{)}x^{-s}\bigg{(}s-\frac{1}{2}\bigg{)}^{-3}\frac{{\rm d}s}{2\pi i}=\frac{M^{1/2}J\Big{(}\mu,...,\mu;\frac{x}{M|\mathfrak{d}_{K/F}|_{\mathbb{R}}}\Big{)}}{|\mathfrak{d}_{K/F}|_{\mathbb{R}}^{1/2}} $$
where $\mu$ appears repeatedly $2n$ times. Then by Proposition
\ref{prop:measures}, Lemma \ref{lem:measures}, and
(\ref{eq:esti-diri}) we obtain {\allowdisplaybreaks
\begin{eqnarray*}
|J(U^{\ast})|&\le& \sum_{\big{|}N_{K/\mathbb{Q}}\big{(}\frac{\mathcal{L}_{i}\mathcal{L}_{j}\mathfrak{o}_{K}}{\mathfrak{N}_{i}\mathfrak{N}_{j}}\big{)}\big{|}\ge U}\frac{M^{1/2}J\bigg{(}\mu,...,\mu,\mu_{F},\mu_{F};\frac{1}{M|\mathfrak{d}_{K/F}|_{\mathbb{R}}}\Big{|}N_{K/\mathbb{Q}}\Big{(}\frac{\mathcal{L}_{i}\mathcal{L}_{j}\mathfrak{o}_{K}}{\mathfrak{N}_{i}\mathfrak{N}_{j}}\Big{)}\Big{|}\bigg{)}}{|\mathfrak{d}_{K/F}|_{\mathbb{R}}^{1/2}}\\
&+& 2\sum_{i=1}^{h}\frac{M^{1/2}J\bigg{(}\mu,...,\mu,\mu_{F},\mu_{K};\frac{1}{M|\mathfrak{d}_{K/F}|_{\mathbb{R}}}\Big{|}N_{K/\mathbb{Q}}\Big{(}\frac{\mathcal{L}_{i}\mathfrak{o}_{K}}{\mathfrak{N}_{i}}\Big{)}\Big{|}\bigg{)}}{|\mathfrak{d}_{K/F}|_{\mathbb{R}}^{1/2}}\\
    &+&\frac{M^{1/2}J\Big{(}\mu,...,\mu,\mu_{K},\mu_{K};\frac{1}{M|\mathfrak{d}_{K/F}|_{\mathbb{R}}}\Big{)}}{|\mathfrak{d}_{K/F}|_{\mathbb{R}}^{1/2}}\\
    &\le& \sum_{\big{|}N_{K/\mathbb{Q}}\big{(}\frac{\mathcal{L}_{i}\mathcal{L}_{j}\mathfrak{o}_{K}}{\mathfrak{N}_{i}\mathfrak{N}_{j}}\big{)}\big{|}\ge U}\frac{M^{1/2}J\bigg{(}\mu,...,\mu,\mu_{F}^{\prime},\mu_{F}^{\prime};\frac{1}{M|\mathfrak{d}_{K/F}|_{\mathbb{R}}}\Big{|}N_{K/\mathbb{Q}}\Big{(}\frac{\mathcal{L}_{i}\mathcal{L}_{j}\mathfrak{o}_{K}}{\mathfrak{N}_{i}\mathfrak{N}_{j}}\Big{)}\Big{|}\bigg{)}}{|\mathfrak{d}_{K/F}|_{\mathbb{R}}^{1/2}}\\
    &+& 2\sum_{i=1}^{h}\frac{M^{1/2}J\bigg{(}\mu,...,\mu,\mu_{F},\mu_{K}^{\prime};\frac{1}{M|\mathfrak{d}_{K/F}|_{\mathbb{R}}}\Big{|}N_{K/\mathbb{Q}}\Big{(}\frac{\mathcal{L}_{i}\mathfrak{o}_{K}}{\mathfrak{N}_{i}}\Big{)}\Big{|}\bigg{)}}{|\mathfrak{d}_{K/F}|_{\mathbb{R}}^{1/2}}\\
    &+&\frac{M^{1/2}J\Big{(}\mu,...,\mu,\mu_{K}^{\prime},\mu_{K}^{\prime};\frac{1}{M|\mathfrak{d}_{K/F}|_{\mathbb{R}}}\Big{)}}{|\mathfrak{d}_{K/F}|_{\mathbb{R}}^{1/2}}\\
    &=&h^{2}h_{F}^{2}A_{2}\int_{\sigma-i\infty}^{\sigma+i\infty}|\mathfrak{d}_{K/F}|_{\mathbb{R}}^{s-1/2}\gamma\bigg{(}s+\frac{1}{2}\bigg{)}s^{2}U^{-s}\bigg{(}s-\frac{1}{2}\bigg{)}^{-5}\frac{{\rm d}s}{2\pi i}\\
    &+&2h_{K}\int_{\sigma-i\infty}^{\sigma+i\infty}|\mathfrak{d}_{K/F}|_{\mathbb{R}}^{s-1/2}\gamma\bigg{(}s+\frac{1}{2}\bigg{)}\zeta_{F}(2s)\sum_{i=1}^{h}\Big{|}N_{K/\mathbb{Q}}\Big{(}\frac{\mathcal{L}_{i}\mathfrak{o}_{K}}{\mathfrak{N}_{i}}\Big{)}\Big{|}^{-s}\\
    &&\ \ \ \ \ \ \ \ \ \ \ \times\bigg{(}s-\frac{1}{2}\bigg{)}^{-3}\times\frac{A_{1}2^{ns-n}s}{(s-1)\sqrt{|\mathfrak{d}_{K/F}|_{\mathbb{R}}}^{s}}\frac{{\rm d}s}{2\pi i}
  \\ &+& h_{K}^{2}\int_{\sigma-i\infty}^{\sigma+i\infty}|\mathfrak{d}_{K/F}|_{\mathbb{R}}^{s-1/2}\gamma\bigg{(}s+\frac{1}{2}\bigg{)}\bigg{(}s-\frac{1}{2}\bigg{)}^{-3}\times\bigg{(}\frac{A_{1}2^{ns-n}s}{(s-1)\sqrt{|\mathfrak{d}_{K/F}|_{\mathbb{R}}}^{s}}\bigg{)}^{2}\frac{{\rm d}s}{2\pi i}\\
  &=:&J_{1}+J_{2}+J_{3}.
\end{eqnarray*} }
By exactly the same argument as the proof of \cite[Proposition 1]{JO
84}, we have
$$
    J_{1}\le\frac{A_{2}h_{K}^{2}h_{F}^{2}M}{96\sqrt{U}}P\bigg{(}{\rm log}\frac{M|\mathfrak{d}_{K/F}|_{\mathbb{R}}}{U}\bigg{)},\\
$$$$     J_{2}\le 16h_{K}A_{1}\gamma\Big{(}\frac{3}{2}\Big{)}\zeta_{F}(2)\sum_{i=1}^{h}\Big{|}N_{K/\mathbb{Q}}\Big{(}\frac{\mathcal{L}_{i}\mathfrak{o}_{K}}{\mathfrak{N}_{i}}\Big{)}\Big{|}^{-1}\\
$$    and $$\ J_{3}\le4A_{1}^{2}{\rm e}M^{3/2}{\rm
max}\big{\{}2, {\rm
log}(4^{n}M)\big{\}}\frac{h_{K}^{2}}{\sqrt{|\mathfrak{d}_{K/F}|_{\mathbb{R}}}}
$$
where $P(T)=T^{4}+a_{3}T^{3}+a_{2}T^{2}+a_{1}T+a_{0}$ denotes the polynomial
\begin{eqnarray*}
    &&T^{4}+[(\varGamma^{2n})^{\prime}(1)+16]T^{3}+[(\varGamma^{2n})^{\prime\prime}(1)+12(\varGamma^{2n})^{\prime}(1)+64]T^{2}\\
    &+&[(\varGamma^{2n})^{\prime\prime\prime}(1)+8(\varGamma^{2n})^{\prime\prime}(1)+40(\varGamma^{2n})^{\prime}(1)]T\\
    &+&[(\varGamma^{2n})^{\prime\prime\prime\prime}(1)+4(\varGamma^{2n})^{\prime\prime\prime}(1)+24(\varGamma^{2n})^{\prime\prime}(1)].
\end{eqnarray*}
Moreover, by calculation we obtain $$ |a_{j}|\le C_{4}^{j}(2n+6)^{4-j}\hskip
30pt  \forall\ \ 0\le j\le3. $$ Then
$$J_{1}\le \frac{A_{2}h_{F}^{2}h^{2}_{K}M}{96\sqrt{U}}\bigg{(}\bigg{|}{\rm log}\frac{M|\mathfrak{d}_{K/F}|_{\mathbb{R}}}{U}\bigg{|}+2n+6\bigg{)}^{4}.$$
By (\ref{eq:estimation of M}), $M\le M^{\prime}=(2\pi)^{-2n}|\mathfrak{ad}^{2}|_{\mathbb{R}}$. Then proposition follows.
\end{proof}

\section{Proof of Theorem \ref{thm:main0}} \label{sec:proof}

In this section, we assume that $|\mathfrak{d}_{K/F}|_{\mathbb{R}}>4^{n}$ ($n=[F:\BQ]$) and $\mathfrak{o}^{\times}_{F}=\mathfrak{o}_{K}^{\times}$. Notice that there are only finitely many $K/F$ with either $|\mathfrak{d}_{K/F}|_{\mathbb{R}}\le4^{n}$ or $\mathfrak{o}_{F}^{\times}\ne\mathfrak{o}_{K}^{\times}$, which can be computed effectively (Remark \ref{rem:extensions by units}).

We put
$$m=\mathrm{max}\bigg{\{}\mathrm{min}\Big{\{}r\in\mathbb{N}\ |\
2r^{2}+2r\ge[\mathfrak{o}^{\times}:(\mathfrak{o}^{\times})^{2}]h_{K}\Big{\}},\frac{3}{2}\bigg{\}},$$
and let
$$ V=\big{(}\frac{|\mathfrak{d}_{K/F}|_{\mathbb{R}}}{4^{n}}\big{)}^{1/h} \ \ \text{and} \ \ U=\bigg{(}\frac{\sqrt{|\mathfrak{d}_{K/F}|_{\mathbb{R}}}}{2^{n}}\bigg{)}^{1/m}>1.$$

 \begin{lem} \label{Norms of primes} We have the following.
    \begin{enumerate}
        \item\label{it:splitting prime-1} For any prime $\mathfrak{p}$ of $F$ splitting in $K$, we have
        $$|\mathfrak{p}|_{\mathbb{R}}\ge V.$$
        \item\label{it:splitting prime-2} There exists at most one prime $\mathfrak{p}$ of $F$ splitting in $K$ such that $|\mathfrak{p}|_{\mathbb{R}}<U$.
        \item\label{it:splitting prime-3} We have
        $$ R:=\max\{|\mathfrak{q}|_{\mathbb{R}}: \mathfrak{q} \text{ is a prime divisor of } \mathfrak{d}_{K/F} \} \ge U.  $$
    \end{enumerate}
\end{lem}
 \begin{proof}
    We first prove (\ref{it:splitting prime-1}). Write $K=F(\sqrt{-a})$ with $a$ totally positive, and write $\mathfrak{o}_{K}=\mathfrak{o}_{F}\oplus\mathfrak{b}\beta$ as $\mathfrak{o}_{F}$-modules where $\beta=x+y\sqrt{-a}\ (x,y\in F)$. By a fundamental calculation, we have
    $$\mathfrak{d}_{K/F}=(\beta-\overline{\beta})^{2}\mathfrak{b}^{2}=(4ay^{2})\mathfrak{b}^{2},$$
      and then
      $|\mathfrak{d}_{K/F}|_{\mathbb{R}}=4^{n}|ay^{2}\mathfrak{b}^{2}|_{\mathbb{R}}$.

      Let $\mathfrak{P}$ be a prime of $K$ lying over $\mathfrak{p}$. Then $\mathfrak{P}^{h}=\mathfrak{a}(\gamma)\mathfrak{o}_{K}$ for some fractional ideal $\mathfrak{a}$ of $F$ and element $\gamma$ of $K^{\ast}$. Write $\gamma=z+w\beta$ ($=(z+wx)+wy\sqrt{-a}$) with $z\in\mathfrak{a}^{-1}$ and $w\in\mathfrak{ba}^{-1}$. We claim that $w\ne0$. Indeed, if $w=0$, we have $\gamma\in F$. Then the condition $\mathfrak{P}^{h}=\mathfrak{a}(\gamma)\mathfrak{o}_{K}$ implies that $\mathfrak{a}(\gamma)\mathfrak{o}_{F}$ is a power of $\mathfrak{p}$. So $\mathfrak{p}$ is not split in $K$, which contradicts to the assumption. By the claim we have
      $|w|_{\mathbb{R}}\ge |\mathfrak{b}|_{\mathbb{R}}|\mathfrak{a}|^{-1}_{\mathbb{R}}$. Then
      $$|\mathfrak{p}|_{\mathbb{R}}^{h}=|\mathfrak{P}|_{\mathbb{R}}^{h}=|\mathfrak{a}|^{2}|(z+wx)^{2}+aw^{2}y^{2}|_{\mathbb{R}}\ge|aw^{2}y^{2}\mathfrak{a}^{2}|_{\mathbb{R}}\ge |ay^{2}\mathfrak{b}^{2}|_{\mathbb{R}}=\frac{|\mathfrak{d}_{K/F}|_{\mathbb{R}}}{4^{n}}.$$
      This proves (\ref{it:splitting prime-1}).

      To prove (\ref{it:splitting prime-2}), for any two splitting primes $\mathfrak{p}$ and $\mathfrak{p}^{\prime}$, we have
      $$\sum_{n}a_{n}n^{-s}:=\frac{1+|\mathfrak{p}|_{\mathbb{R}}^{-s}}{1-|\mathfrak{p}|_{\mathbb{R}}^{-s}}\cdot\frac{1+|\mathfrak{p}^{\prime}|_{\mathbb{R}}^{-s}}{1-|\mathfrak{p}^{\prime}|_{\mathbb{R}}^{-s}}\ll\frac{\zeta_{K}(s)}{\zeta_{F}(2s)}. $$ Suppose both $\mathfrak{p}$ and $\mathfrak{p}^{\prime}$ have norm $\le U$. Then for any $r,s\in\mathbb{Z}_{\ge0}$ such that $r+s\le [m]$, we have $|\mathfrak{p}|_{\mathbb{R}}^{r}|\mathfrak{p}^{\prime}|_{\mathbb{R}}^{s}<\sqrt{|\mathfrak{d}_{K/F}|_{\mathbb{R}}}/2^{n}$. Here $[x]$ denotes the integer part of $x$. So we have
      $$\sum_{n<\frac{\sqrt{|\mathfrak{d}_{K/F}|_{\mathbb{R}}}}{2^{n}}}a_{n}\ge2[m]^{2}+2[m]+1>[\mathfrak{o}^{\times}:(\mathfrak{o}^{\times})^{2}]h_{K}> h,$$
      which contradicts to Proposition \ref{prop:estimation of h} (\ref{it:estimation-3}). This proves (\ref{it:splitting prime-2}).

      To prove (\ref{it:splitting prime-3}), let $t$ be the number of prime divisors of $\mathfrak{d}_{K/F}$. By Corollary \ref{cor:estimate h by t}, we have $2^{t}\le[\mathfrak{o}^{\times}:(\mathfrak{o}^{\times})^{2}]h_{K}\le 2m^{2}+2m$. Then $t\le 2m$. Let $\{\mathfrak{p}_{1},\cdots,\mathfrak{p}_{r}\}$ be the set of prime divisors of $\mathrm{g.c.d}(2,\mathfrak{d}_{K/F})$.
    By Corollary \ref{cor: prime power in discriminant }, for each $i=1,\cdots,r$, there exists $\ell_{i}$  ($1\le\ell_{i}\le 2e_{i}$) such that $\mathfrak{d}_{K/F}\mathfrak{p}_{1}^{-\ell_{1}}\cdots\mathfrak{p}_{r}^{-\ell_{r}}$
    has no square factors. We write $f_{i}$ for the inertial degree of
    $\mathfrak{p}_{i}$.  Then
    $$R^{2m}\ge \bigg{|}\frac{\mathfrak{d}_{K/F}}
    {\mathfrak{p}_{1}^{\ell_{1}}\cdots\mathfrak{p}_{r}^{\ell_{r}}}\bigg{|}_{\mathbb{R}}=\frac{|\mathfrak{d}_{K/F}|_{\mathbb{R}}}{2^{\ell_{1}f_{1}+\cdots\ell_{r}f_{r}}}
    \ge\frac{|\mathfrak{d}_{K/F}|_{\mathbb{R}}}{4^{e_{1}f_{1}+\cdots e_{r}f_{r}}}\ge\frac{|\mathfrak{d}_{K/F}|_{\mathbb{R}}}{4^{n}}.$$
    This proves (\ref{it:splitting prime-3}).
\end{proof}

Put
$$P(U,K):=\{\text{prime divisors}\ \mathfrak{p}\ \text{of}\ \mathfrak{d}_{K/F}:\ |\mathfrak{p}|_{\mathbb{R}}<U\}$$
and
$$ P(K):=\{\text{prime divisors}\ \mathfrak{p}\ \text{of}\ \mathfrak{d}_{K/F}:\  |\mathfrak{p}|_{\mathbb{R}}< R\}. $$
By Lemma \ref{Norms of primes} (\ref{it:splitting prime-3}) we have
$P(U,K)\subseteq P(K)$.

\begin{lem}(\cite[(3.5.1)-(3.5.1)]{JO 84})\label{lem: estimation of Phi(s)}
Let $\varPsi(s)$ and $\varPhi(s)$ satisfy Condition \ref{condition}. Suppose $|\mathfrak{d}_{K/F}|_{\mathbb{R}}>4^{n}$. Then we have
 \begin{equation}\label{ineq1}
 	 1\le\sum_{i=1}^{h}N_{K/\mathbb{Q}}\bigg{|}\bigg{(}\frac{\mathcal{L}_{i}\mathfrak{o}_{K}}{\mathfrak{N}_{i}}\bigg{)}\bigg{|}^{-1}\le   D_{1}(K)\cdot\prod_{\mathfrak{p}\in P(U,K)}(1+|\mathfrak{p}|_{\mathbb{R}}^{-1})
 	 \end{equation}
  \begin{equation}\label{ineq2}
  	|\varPhi(U,1)|\ge D_{2}(K)\cdot\prod_{\mathfrak{p}\in
  		P(U,K)}\varPhi_{\mathfrak{p}}(1)
  \end{equation}
  \begin{equation}\label{ineq3}
  	{\rm Re}\bigg{(}\frac{\varPhi^{\prime}(U,1)}{\varPhi(U,1)}\bigg{)}\ge\sum_{\mathfrak{p}\in P(U,K)}{\rm Re}\bigg{(}\frac{\varPhi^{\prime}_{\mathfrak{p}}(1)}{\varPhi_{\mathfrak{p}}(1)}\bigg{)}-D_{3}(K)
  \end{equation}
and
\begin{equation}\label{ineq4}
\hskip20pt{\rm sup}_{s\in\varDelta}\bigg{|}\frac{\varPhi(U,s+\frac{1}{2})}{\varPhi(U,1)}\bigg{|}\le D_{4}(K)\cdot\prod_{\mathfrak{p}\in P(U,K)}(1-|\mathfrak{p}|_{\mathbb{R}}^{-1/4})^{-2}.	
\end{equation}
Here
$$D_{1}(K)=\bigg{(}1+\frac{h}{U}\bigg{)}\frac{1+V^{-1}}{1-V^{-1}},\ D_{2}(K)=\bigg{(}\frac{1-V^{-1/2}}{1+V^{-1/2}}\bigg{)}^{2},\ D_{3}(K)=\frac{4V^{-1/2}}{1-V^{-1/2}}{\rm log}U$$ and
$$ D_{4}(K)=\frac{(1+V^{-1/2})^{2}}{(1-V^{-1/4})^{4}}.  $$
\end{lem}

\begin{proof}
First we prove (\ref{ineq1}). By Propisition \ref{prop:estimation of h} (\ref{it:estimation-2}),
$$\sum_{i=1}^{h}\bigg{|}N_{K/\mathbb{Q}}\bigg{(}\frac{\mathcal{L}_{i}\mathfrak{o}_{K}}{\mathfrak{N}_{i}}\bigg{)}\bigg{|}^{-s}\le\frac{\zeta_{K}(s)}{\zeta_{F}(2s)}=\prod_{\text{$\mathfrak{p}$ splits in $K$}}\frac{1+|\mathfrak{p}|_{\mathbb{R}}^{-s}}{1-|\mathfrak{p}|_{\mathbb{R}}^{-s}}\times\prod_{\text{$\mathfrak{p}$ ramifies in $K$}}(1+|\mathfrak{p}|_{\mathbb{R}}^{-s}).         $$
For each $i\in I$, let $k_{i}$ be the number of prime divisors (with multiplicity) of $N_{K/\mathbb{Q}}\big{(}\frac{\mathcal{L}_{i}\mathfrak{o}_{K}}{\mathfrak{N}_{i}}\big{)}$ that have absolute norms $\ge U$. Put $I_{1}=\{i\in\{1,\cdots,h\}:k_{i}>0\}$ and $I_{2}=\{1,\cdots,h\}\setminus I_{1}$. By Lemma \ref{Norms of primes} we have
$$\sum_{i\in I_{1}}\bigg{|}N_{K/\mathbb{Q}}\bigg{(}\frac{\mathcal{L}_{i}\mathfrak{o}_{K}}{\mathfrak{N}_{i}}\bigg{)}\bigg{|}^{-1}\le\sum_{i\in I_{1}}\frac{1}{U^{k_{i}}}\cdot\frac{1+V^{-1}}{1-V^{-1}}\prod_{\mathfrak{p}\in P(U,K)}(1+|\mathfrak{p}|_{\mathbb{R}}^{-1}) $$ and
$$\sum_{i\in I_{2}}\bigg{|}N_{K/\mathbb{Q}}\bigg{(}\frac{\mathcal{L}_{i}\mathfrak{o}_{K}}{\mathfrak{N}_{i}}\bigg{)}\bigg{|}^{-1}\le \frac{1+V^{-1}}{1-V^{-1}}\prod_{\mathfrak{p}\in P(U,K)}(1+|\mathfrak{p}|_{\mathbb{R}}^{-1}), $$ which deduce (\ref{ineq1}).

By Condition \ref{condition} (\ref{it-Dirichlet-1}), using the Euler product of $\frac{\zeta_{K}(s)}{\zeta_{F}(2s)}$ we get
	$$|\varPhi(U,1)|=\prod_{\text{$\mathfrak{p}$ splits in $K$}, |\mathfrak{p}|_{\mathbb{R}}<U}|\varPhi_{\mathfrak{p}}(1)|\times\prod_{\mathfrak{p}\in P(U,K)}|\varPhi_{\mathfrak{p}}(1)|$$
	$$\frac{\varPhi^{\prime}(U,1)}{\varPhi(U,1)}=\sum_{\text{$\mathfrak{p}$ splits in $K$},|\mathfrak{p}|_{\mathbb{R}}<U}\frac{\varPhi_{\mathfrak{p}}^{\prime}(1)}{\varPhi_{\mathfrak{p}}(1)}+\sum_{\mathfrak{p}\in P(U,K)}\frac{\varPhi_{\mathfrak{p}}^{\prime}(1)}{\varPhi_{\mathfrak{p}}(1)}$$ 
	and
$${\rm sup}_{s\in\varDelta}\bigg{|}\frac{\varPhi(U,s+\frac{1}{2})}{\varPhi(U,1)}\bigg{|}\le\prod_{\text{$\mathfrak{p}$ splits in $K$}, |\mathfrak{p}|_{\mathbb{R}}<U}{\rm sup}_{s\in\varDelta}\bigg{|}\frac{\varPhi_{\mathfrak{p}}(U,s+\frac{1}{2})}{\varPhi_{\mathfrak{p}}(U,1)}\bigg{|}\times\prod_{\mathfrak{p}\in P(U,K)}{\rm sup}_{s\in\varDelta}\bigg{|}\frac{\varPhi_{\mathfrak{p}}(U,s+\frac{1}{2})}{\varPhi_{\mathfrak{p}}(U,1)}\bigg{|}.$$
 Then (\ref{ineq2}), (\ref{ineq3}) and (\ref{ineq4}) are deduced from Condition \ref{condition} (\ref{it-Dirichlet-2}) as well as (\ref{it:splitting prime-1}), (\ref{it:splitting prime-2}) of Lemma \ref{Norms of primes}.

\end{proof}

\begin{thm}\label{thm:main}
Fix a totally real number field $F$. Let $\mathbf{f}$ be a Hilbert newform of parallel weight 2 with level $\mathfrak{a}$. For every totally imaginary extension $K$ of $F$, we let $\chi_{K}=(\ \ ,K/F)$.  Put
$$\mathcal{K}:=\{\text{quadratic imaginary extensions $K$ of $F$}|\ \text{$(\mathbf{f},\chi_{K})$ satisfies (\ref{it-modular form-2}), (\ref{it-modular form-3}) of Lemma \ref{lem:hilbert modular form}} \}.   $$
We write $\varPsi_{K}(s)$ (resp. $\varPhi_{K}(s)$) for $\varPsi(s)$ (resp. $\varPhi(s)$) that is obtained from Lemma \ref{lem:hilbert modular form}. Then there exists an effective constant $C$, which depends only on $\mathbf{f}$ such that for every $K\in\mathcal{K}$ with $|\mathfrak{d}_{K/F}|_{\mathbb{R}}>4^{n}$ and $\mathfrak{o}_{K}^{\times}=\mathfrak{o}_{F}^{\times}$,
$$ h_{K}\ge C\cdot\prod_{\mathfrak{p}\in P(K)}\bigg{(}1-\frac{2\sqrt{|\mathfrak{p}|_{\mathbb{R}}}}{1+|\mathfrak{p}|_{\mathbb{R}}}\bigg{)} {\rm log}|\mathfrak{d}_{K/F}|_{\mathbb{R}}.$$
\end{thm}
\begin{proof}
The argument is similar to the proof of \cite[TH\'EOR\`EME 2]{JO
84}. We will show the effectiveness of the constant $C$. We fix a $K\in\mathcal{K}$, and we may assume that $|\mathfrak{d}_{K/F}|_{\mathbb{R}}>4^{n}$. Let $\lambda$ be any positive number. If $|\mathfrak{d}_{K/F}|_{\mathbb{R}}<{\rm e}^{\lambda h_{K}}$, we have $h_{K}>\lambda^{-1}{\rm log}|\mathfrak{d}_{K/F}|_{\mathbb{R}}$. 

From now on, we suppose $|\mathfrak{d}_{K/F}|_{\mathbb{R}}\ge{\rm e}^{\lambda h_{K}}$. Then clearly $V\ge\frac{{\rm exp}(\lambda)}{4^{n}}$. Moreover, by the definitions of $m$ and $U$, it's easy to see that
 \begin{equation}\label{eq:estimations of m,U}
 	m\le3\sqrt{[\mathfrak{o}^{\times}:(\mathfrak{o}^{\times})^{2}]h_{K}}\hskip10pt\text{and then}\hskip10pt U\ge \frac{\lambda^{16} h_{K}^{8}}{16!\cdot3^{16}\cdot 2^{\frac{2n}{3}+16}[\mathfrak{o}^{\times}:(\mathfrak{o}^{\times})^{2}]^{8}}.
 \end{equation}
  Put $$D_{1}(\lambda)=\bigg{(}1+\frac{16!\cdot3^{16}\cdot2^{\frac{2n}{3}+16}[\mathfrak{o}^{\times}:(\mathfrak{o}^{\times})^{2}]^{8}}{\lambda^{16}}\bigg{)}\frac{{\rm e}^{\lambda}+4^{n}}{{\rm e}^{\lambda}-4^{n}},\hskip10ptD_{2}(\lambda)=\bigg{(}\frac{{\rm e}^{\lambda/2}-2^{n}}{{\rm e}^{\lambda/2}+2^{n}}\bigg{)}^{2},\hskip10pt D_{3}(\lambda)=\frac{2^{n+2}}{{\rm e}^{\lambda/2}-2^{n}}$$
 and
  $$D_{4}(\lambda)=\frac{({\rm e}^{\lambda/2}+2^{n})^{2}}{({\rm e}^{\lambda/4}-2^{n/2})^{4}} .$$ Then we have $D_{1}(K)\le D_{1}(\lambda)$, $D_{2}(K)\ge D_{2}(\lambda)$, $D_{3}(K)\le D_{3}(\lambda){\rm log}U$ and $D_{4}(K)\le D_{4}(\lambda)$.
 
 First we consider the upper bound for $|J(U^{\ast})|$. Indeed, we have
\begin{eqnarray*}
	&&\frac{B_{1}h_{K}}{\sqrt{U}}\bigg{(}\bigg{|}{\rm log}\frac{M^{\prime}|\mathfrak{d}_{K/F}|_{\mathbb{R}}}{U}\bigg{|}+2n+6\bigg{)}^{4}\\
	&\le&\frac{B_{1}h_{K}}{\sqrt{U}}\big{(}64m\cdot{\rm log}(2^{\frac{n}{32m}}U^{\frac{1}{32}-\frac{1}{64m}})+|{\rm log}M^{\prime}|+2n+6\big{)}^{4}\\
	&\le&\bigg{(}\frac{B_{1}^{\frac{1}{4}}h_{K}^{\frac{1}{4}}\cdot64m\cdot2^{\frac{n}{48}}U^{\frac{1}{32}}}{U^{\frac{1}{8}}}+\frac{B_{1}^{\frac{1}{4}}h_{K}^{\frac{1}{4}}(|{\rm log}M^{\prime}|+2n+6)}{U^{\frac{1}{8}}}\bigg{)}^{4}\\
	&\le&\bigg{(}\frac{2^{\frac{n+90}{12}}\cdot3^{\frac{5}{2}}(16!)^{\frac{3}{32}}[\mathfrak{o}^{\times}:(\mathfrak{o}^{\times})^{2}]^{\frac{5}{4}}B_{1}^{\frac{1}{4}}}{\lambda^{3/2}}+\frac{2^{\frac{n+24}{48}}3^{\frac{1}{2}}(16!)^{\frac{1}{32}}[\mathfrak{o}^{\times}:(\mathfrak{o}^{\times})^{2}]^{\frac{1}{4}}B_{1}^{\frac{1}{4}}(|{\rm log}M^{\prime}|+2n+6)}{\lambda^{1/2}}\bigg{)}^{4}\\
	&(=: &F_{1}(\lambda)),
\end{eqnarray*}
where in the last step, we have used (\ref{eq:estimations of m,U}). If we put
$$E_{1}(\lambda)=F_{1}(\lambda)+B_{2}D_{1}(\lambda)+\frac{2B_{3}}{\lambda},$$
then by Lemma \ref{lem: J(U) and J(U*)} (\ref{it-J(U)-1}), Proposition \ref{prop: estimation of J(U*)} and Lemma \ref{lem:
estimation of Phi(s)} (\ref{ineq1}):
\begin{eqnarray*}
\scriptstyle|J(U)|=|J(U^{\ast})|&\le&\scriptstyle \Big{(}\frac{B_{1}h_{K}}{\sqrt{U}}\Big{(}\Big{|}{\rm log}\frac{M^{\prime}|\mathfrak{d}_{K/F}|_{\mathbb{R}}}{U}\Big{|}+2n+6\Big{)}^{4}+B_{2}D_{1}(K)+B_{3}\frac{h_{K}}{\sqrt{|\mathfrak{d}_{K/F}|_{\mathbb{R}}}}\Big{)}h_{K}\cdot\prod_{\mathfrak{p}\in P(U,K)}(1+|\mathfrak{p}|_{\mathbb{R}}^{-1})\\
&\le& \scriptstyle E_{1}(\lambda)h_{K}\cdot\prod_{\mathfrak{p}\in P(U,K)}(1+|\mathfrak{p}|_{\mathbb{R}}^{-1}).	
\end{eqnarray*}

Next we consider the lower bound for $|J(U)|$. For each prime divisor $\mathfrak{p}$ of $\mathfrak{d}_{K/F}$, by the construction of Lemma \ref{lem:hilbert modular form} as well as Corollary \ref{cor:estimate h by t}, we have 
$$	\bigg{|}\sum_{\mathfrak{p}\in P(U,K)}{\rm Re}\bigg{(}\frac{\varPhi^{\prime}_{K,\mathfrak{p}}(1)}{\varPhi_{K,\mathfrak{p}}(1)}\bigg{)}\bigg{|}\le8\#P(U,K)\le 8{\rm log}_{2}([\mathfrak{o}^{\times}:(\mathfrak{o}^{\times})^{2}]h_{K})
$$ 
and
$$\prod_{\mathfrak{p}\in
	P(U,K)}(1-|\mathfrak{p}|_{\mathbb{R}}^{-1/4})^{-2}\le F_{2}\cdot2^{\#P(U,K)}\le  F_{2}\cdot[\mathfrak{o}^{\times}:(\mathfrak{o}^{\times})^{2}]h_{K}.$$ 
Here $$F_{2}=\prod_{\mathfrak{p}\in P(U,K),|\mathfrak{p}|_{\mathbb{R}}\le131}\frac{\sqrt{|\mathfrak{p}|_{\mathbb{R}}}}{(|\mathfrak{p}|_{\mathbb{R}}^{1/4}-1)^{2}}$$ 
is a constant. We put $$E_{2}(\lambda)= 1-\Big{(}\frac{8[\mathfrak{o}^{\times}:(\mathfrak{o}^{\times})^{2}]{\rm log}_{2}({\rm e})}{\lambda}+\frac{D_{3}(\lambda)}{3}+\frac{G_{2}}{\lambda}+\frac{G_{3}D_{4}(\lambda)F_{2}}{G_{1}\lambda}[\mathfrak{o}^{\times}:(\mathfrak{o}^{\times})^{2}]\Big{)}.$$ Then $\lim_{\lambda\to\infty}E_{2}(\lambda)=1$. By Lemma \ref{lem: J(U) and J(U*)}
(\ref{it-J(U)-2}) and Lemma \ref{lem: estimation of Phi(s)} (\ref{ineq2}), (\ref{ineq3}) and (\ref{ineq4}), we have
\begin{eqnarray*}
&&\scriptstyle|J(U)|\\
&\ge& \scriptstyle|G_{1}(M,\varPsi)||\varPhi(U,1)|\big{(}{\rm log|\mathfrak{d}_{K/F}|_{\mathbb{R}}}+{\rm Re}\big{(}\frac{\varPhi^{\prime}_{K}(U,1)}{\varPhi_{K}(U,1)}\big{)}+G_{2}(M,\varPsi)-\frac{G_{3}(M,\varPsi)}{|G_{1}(M,\varPsi)|}{\rm sup}_{s\in\varDelta}\big{|}\frac{\varPhi_{K}(U,s+\frac{1}{2})}{\varPhi_{K}(U,1)}\big{|}\big{)}\\
&\ge&\scriptstyle G_{1}D_{2}(\lambda)\Big{(}1-\Big{(}8{\rm log}_{2}([\mathfrak{o}^{\times}:(\mathfrak{o}^{\times})^{2}]h_{K})+D_{3}(\lambda){\rm log}U+G_{2}+\frac{G_{3}D_{4}(\lambda)F_{2}}{G_{1}}[\mathfrak{o}^{\times}:(\mathfrak{o}^{\times})^{2}]h_{K}\Big{)}({\rm log}|\mathfrak{d}_{K/F}|)^{-1}\Big{)}\\
&&\hskip250pt\scriptstyle\times\big{|}\prod_{\mathfrak{p}\in
P(U,K)}\varPhi_{K,\mathfrak{p}}(1)\big{|}\cdot{\rm
log}|\mathfrak{d}_{K/F}|_{\mathbb{R}}\\
&\ge&\scriptstyle G_{1}D_{2}(\lambda)E_{2}(\lambda)\cdot\big{|}\prod_{\mathfrak{p}\in
    P(U,K)}\varPhi_{K,\mathfrak{p}}(1)\big{|}\cdot{\rm
    log}|\mathfrak{d}_{K/F}|_{\mathbb{R}}.
\end{eqnarray*}
Moreover, one can verify that
$$    \Big{|}\frac{\varPhi_{\mathfrak{p}}(1)}{1+|\mathfrak{p}|_{\mathbb{R}}^{-1}}\Big{|}\ge1-\frac{2\sqrt{|\mathfrak{p}|_{\mathbb{R}}}}{1+|\mathfrak{p}|_{\mathbb{R}}}$$ for every $\mathfrak{p}\mid\mathfrak{d}_{K/F}.$ Therefore we have
\begin{eqnarray*}
	h_{K}&\ge&\frac{G_{1}D_{2}(\lambda)E_{2}(\lambda)}{E_{1}(\lambda)}\bigg{(}\prod_{\mathfrak{p}\in P(U,K)}\bigg{|}\frac{\varPhi_{K,\mathfrak{p}}(1)}{1+|\mathfrak{p}|_{\mathbb{R}}^{-1}}\bigg{|}\bigg{)}{\rm log}|\mathfrak{d}_{K/F}|_{\mathbb{R}}\\
	&\ge&\frac{G_{1}D_{2}(\lambda)E_{2}(\lambda)}{E_{1}(\lambda)}\prod_{\mathfrak{p}\in P(K)}\bigg{(}1-\frac{2\sqrt{|\mathfrak{p}|_{\mathbb{R}}}}{1+|\mathfrak{p}|_{\mathbb{R}}}\bigg{)} {\rm log}|\mathfrak{d}_{K/F}|_{\mathbb{R}}.
\end{eqnarray*}
For any $\lambda>0$ such that $E_{2}(\lambda)>0$, the constant $$C={\rm min}\bigg{\{}\lambda^{-1},\frac{G_{1}D_{2}(\lambda)E_{2}(\lambda)}{E_{1}(\lambda)}\bigg{\}}$$
satisfies our requirements.
The effectiveness of $C$ follows from the computations in the above, as well as Remark \ref{rem: Hilbert modular form}, Lemma \ref{Lang}, Lemma  \ref{lem:norm estimation}, Lemma \ref{lem: J(U) and J(U*)} and Proposition \ref{prop: estimation of J(U*)}.
This proves the theorem.
\end{proof}

\begin{cor} \label{cor: main}$($$\Rightarrow$ Theorem \ref{thm:main0}$)$
Assume that $F$ satisfies one of the following conditions.
\begin{enumerate}
    \item\label{it-final estimation-1} $[F:\mathbb{Q}]$ is odd.
    \item\label{it-final estimation-2} $F/\mathbb{Q}$ is solvable and $[F:\mathbb{Q}]$ is even. If we write $37\mathfrak{o}_{F}=\mathfrak{p}_{1}^{e}\cdots\mathfrak{p}_{s}^{e}$, then $s$ is even.
\end{enumerate}
Then for every quadratic imaginary extension $K/F$ with $|\mathfrak{d}_{K/F}|_{\mathbb{R}}>4^{n}$ and $\mathfrak{o}_{K}^{\times}=\mathfrak{o}_{F}^{\times}$, we have
$$h_{K}\ge{\rm min}\bigg{\{}\frac{1}{f{\rm log}37}{\rm
	log}\frac{|\mathfrak{d}_{K/F}|_{\mathbb{R}}}{4^{n}},\ C\cdot\prod_{\mathfrak{p}\in P(K)}\bigg{(}1-\frac{2\sqrt{|\mathfrak{p}|_{\mathbb{R}}}}{1+|\mathfrak{p}|_{\mathbb{R}}}\bigg{)} {\rm log}|\mathfrak{d}_{K/F}|_{\mathbb{R}}\bigg{\}}. $$
    \end{cor}
\begin{proof}
If one of $\mathfrak{p}_{1},\cdots,\mathfrak{p}_{s}$ splits in $K$,
by Lemma \ref{Norms of primes} (\ref{it:splitting prime-1}) we have
$$ h_{K}\ge h\ge \frac{1}{f{\rm log}37}{\rm
log}\frac{|\mathfrak{d}_{K/F}|_{\mathbb{R}}}{4^{n}},$$ where $f=n/se$
denotes the inertial degree (over $\mathbb{Q}$) of such a splitting
prime.

If none of $\mathfrak{p}_{1},\cdots,\mathfrak{p}_{s}$ split in $K$,
our assertion follows from Theorem \ref{thm:main} by applying the Dirichlet series constructed in Proposition \ref{prop:epsilon factors}.
Notice that both conditions (\ref{it-final estimation-1}) and
(\ref{it-final estimation-2}) here for $F$ imply that $n+s$ is even. 
\end{proof}

\end{document}